\documentclass[11pt]{article}
\usepackage{amsthm}
\usepackage{amsmath}
\usepackage{amsfonts}
\usepackage{amssymb}
\usepackage[all, cmtip]{xy}
\usepackage{geometry}
\usepackage{mathabx}
\usepackage{hyperref}

 \DeclareFontFamily{U}{mathx}{\hyphenchar\font45}
 \DeclareFontShape{U}{mathx}{m}{n}{<-> mathx10}{}
 \DeclareSymbolFont{mathx}{U}{mathx}{m}{n}
 \DeclareMathAccent{\widebar}{0}{mathx}{"73}

\newtheorem{thm}{Theorem}[section]
\newtheorem{question}[thm]{Question}
\newtheorem{notation}[thm]{Notation}
\newtheorem{remark}[thm]{Remark}
\newtheorem{lemma}[thm]{Lemma}
\newtheorem{proposition}[thm]{Proposition}

\newtheorem{definition}[thm]{Definition}
\newtheorem{theorem}[thm]{Theorem}
\newtheorem{corollary}[thm]{Corollary}

\newcommand{\mA}{\mathbb{A}}
\newcommand{\mZ}{\mathbb{Z}}
\newcommand{\mL}{\mathbb{L}}

\newcommand{\mV}{\mathbb{V}}
\newcommand{\mr}[1]{\ensuremath{\mathrm{#1}}}
\newcommand{\s}[1]{\mathcal{#1}}

\newcommand{\sJ}{\s{J}}
\newcommand{\sH}{\s{H}}
\newcommand{\sK}{\s{K}}

\newcommand{\sV}{\s{V}}
\newcommand{\sO}{\s{O}}
\newcommand{\sF}{\s{F}}
\newcommand{\sG}{\s{G}}
\newcommand{\sT}{\s{T}}
\newcommand{\sP}{\s{P}}
\newcommand{\sI}{\s{I}}
\newcommand{\sM}{\s{M}}
\newcommand{\sS}{\s{S}}
\newcommand{\sE}{\s{E}}

\newcommand{\dsch}{d\mr{Sch}}

\newcommand{\pscoh}{P\mathrm{sCoh}}
\newcommand{\QC}{Q\mathrm{Coh}}
\newcommand{\spec}{\operatorname{Spec}}
\newcommand{\perf}{P\mathrm{erf}}
\newcommand{\bderive}{\mathrm{D}^{b}}
\newcommand{\mP}{\mathbb{P}}
\newcommand{\map}{\mr{Map}}
\newcommand{\dCh}{d\mr{CH}}
\newcommand{\dSch}{d\mr{Sch}}
\newcommand{\image}{\operatorname{Im}}

\newcommand{\Map}{Map}
\newcommand{\cofib}{\operatorname{cofib}}
\newcommand{\Td}{\mr{Td}}
\newcommand{\mQ}{\mathbb{Q}}
\newcommand{\dChQ}{\dCh_{\mQ}}

\newcommand{\mT}{\mathbb{T}}
\newcommand{\dcone}{M^\circ_{\widebar X} Y}

\newcommand{\holimproj}{\underleftarrow{\operatorname{\mathstrut holim}}}

\title{Grothendieck-Riemann-Roch for derived schemes}
\author{Parker Lowrey and Timo Sch\"urg}

\begin{document}

\maketitle
\abstract{We define bivariant algebraic K-theory and bivariant derived Chow on the homotopy category of derived schemes over a smooth base.  The orientation on the latter corresponds to virtual Gysin homomorphisms.  We then provide a morphism between these two bivariant theories and compare the two orientations.  This comparison then yields a homological and cohomological Grothendieck-Riemann-Roch formula for virtual classes.}

\section{Introduction}

In this paper we extend bivariant K-theory and bivariant operational Chow to derived schemes with quasi-projective underlying scheme over a smooth base.  We orient these theories along quasi-smooth, i.e.  ``virtually smooth'', morphisms.  We provide a morphism between the two theories and compare the two orientations.  The result is Grothendieck-Riemann-Roch formulas for quasi-smooth morphisms.   In it's most basic form (over a point and using structure morphisms), it recovers Kontsevich's formula for the virtual class of \cite[Section 1.4.2]{kontsevich}.  In addition to providing useful computational formulas, our results provide an example of how the use of derived schemes can explain much of the ``virtual'' phenomena by extensions of standard methods in classical intersection theory.

Although the context of this paper is in the realm of derived algebraic geometry, if we restrict to the subcategory of quasi-smooth schemes, our results can be treated as standalone formulas and easily applied to classical algebraic geometry through perfect obstruction theories.  For example, we provide a Grothendieck-Riemann-Roch formula for virtually smooth morphisms:
\begin{proposition}
Let $f: X \to Y$ be a proper morphism with $X$, $Y$ quasi-smooth (and quasi-projective) and let $[\sP] \in K([f])$.  Then
 \begin{displaymath}
ch(f_\ast([\sP])) \cap \Td(\mT_{Y/k}) \cap [Y]_{vir} = f_\ast(ch([\sP])\cap \Td(\mT_{X/k}) \cap [X]_{vir})
\end{displaymath}
\end{proposition}
\noindent Further, if $Y$ is smooth, then $[\sP] = \sum (-1)^i[\sH^i(\sP)]$ and the formula is reduced to a computation on the underlying scheme.  Many of these formulas appear in \S\ref{applications}; readers primarily interested in applications to virtual fundamental classes should consult this section.  

Recall that a bivariant theory is the process of assigning graded abelian groups to morphisms in a category in such a fashion that they satisfy many desirable ``cohomological'' and ``homological'' properties.  An operational bivariant theory is the process of taking a algebraic homology theory and formally bootstrapping it to a bivariant theory.  Thus, to define bivariant operational derived Chow, we only need to define the Chow group of a derived scheme, $\dCh(X)$.  Since the ``derived'' data of a derived scheme is isolated in the structure sheaf and the underlying topological space $\widebar{X}$ is a scheme, it is clear that the proper map $\iota_X: \widebar X \hookrightarrow X$ should induce an isomorphism: $\dCh(X) \cong \mr{CH}(\widebar X)$.  This fact ensures our operational bivariant derived Chow is a ``free'' extension of the classical bivariant theory with the advantage of including all quasi-smooth morphisms in the orientation.  For bivariant K-theory,  we extend the notion of relatively perfect as defined in \cite{SGA6} to derived schemes and take the associated Grothendieck group.  If we work over a smooth base Lemma \ref{sameoverafield} guarantees $K(X) = K(\widebar X)$. For reasons similar to the Chow theory, this ensures that the derived ``G-theory'' only picks up underlying topological information.  For general morphisms, the difference between $K([f]) = K([f \circ \iota_X])$ is still unclear.  

With these definitions in place, we use the local Chern character to define a morphism $\tau: K([f]) \to \dCh([f])$ for any morphism $f$ of quasi-projective derived schemes.  We then prove that this morphism is compatible with the bivariant structure.
\begin{theorem}
  $\tau: K \to \dChQ$ is a morphism of bivariant theories.
\end{theorem}

We assign orientations to both of these theories on the class of quasi-smooth morphisms.  In the K-theory case, this orientation is $[\sO_X]$;  for the Chow theory, if $f: X \to Y$ is quasi-smooth the orientation is the virtual Gysin homomorphism $[f^!]_{vir}$ corresponding to the obstruction theory $(f, \iota_X^\ast \mL_{X/Y})$.  The Grothendieck-Riemann-Roch formulas are a result of the following comparison between the image of the K-theory orientation (under $\tau$) and the Chow orientation:
\begin{theorem}
  Let $f: X \to Y$ be quasi-smooth.  Then $\tau([\sO_X]) = \Td(\mT_{X/Y}) \cap [f^!]_{vir}$.
\end{theorem}

The following proposition is the main ingredient in the comparison.  It is a direct result of repackaging virtual information in the structure sheaf, rather than separate data (compare to \cite[Corollary 18.1.2]{fulton_intersection}):
\begin{proposition}
Let $f: X \to Y$ be a  quasi-smooth embedding.  Then $dch_X^Y(\sO_X) = \Td(\mT_{X/Y}) \cap  [f^!]_{vir}$.
\end{proposition}
\noindent The difference between $dch$ and the local Chern character $ch$ is small and can be ignored in the above statement. In fact, compatibility with proper pushforward ensures it is the unique extension of the local Chern character to the derived theory.  

We mention this proposition since its proof illuminates (and motivates) the definition of the virtual Gysin homomorphism:  there exists a deformation of $X$,  $\widehat X \to \dcone $, such that the restriction of $\widehat X$ to the special fiber is equivalent to $\sO_{\widebar X} \otimes^L_{Sym \iota^\ast _X\mL_{X/Y}} \sO_{C_{\widebar X} Y}$.  This shows (in K-theory) that tensoring with $\sO_X$ is equivalent, after applying $\mA^1$-invariance and the projection formula, to the virtual Gysin homomorphism for K-theory, i.e. deforming to the normal cone, pushing forward to the obstruction bundle, and intersecting with the zero section.  Thus, using the invariance of the local Chern character it forces the standard virtual Gysin homomorphism for Chow theory as well.

The restriction to quasi-projectivity is the built in global factorizations they provide.  This significantly reduces complication in defining $\tau$.  There exists a strong possibility of extending parts of this theory past the quasi-projectivity requirement.  At the very least, it should be possible to extend this to certain Deligne-Mumford stacks and more general proper varieties.

In the last section we give some comparisons of our result with that of \cite{kapranov_fontanine} and \cite{fantechi_gottsche}.  It should be noted that the initial motivations for this project were to extend the results in \cite{kapranov_fontanine} to more general classes of morphisms.  In cases that our paper overlaps with the their assumptions, e.g., working with a field of characteristic 0 and fixing a smooth embedding of a quasi-smooth derived scheme into a smooth scheme, we can derive \cite[Theorem 4.2.3]{kapranov_fontanine} as a direct consequence of Theorem \ref{grr} combined with the fact that $[\sO_X] = \sum (-1)^i[\sH^i(\sO_X)] \in K_0(X)$ in this case.   For the comparison with \cite{fantechi_gottsche}, we show that the class $\sum (-1)^i[\sH^i(\sO_X)] = [\sO_{X}^{vir}]$\footnote{as defined in \cite{fantechi_gottsche}} in $K(\widebar X)$.  Combining this with our comparison with \cite{kapranov_fontanine} allows us to derive their result from ours.  From this comparison it becomes immediate that when working with obstruction theories over a point rather than derived schemes, the class $[\sO_{X}^{vir}]$ is the ``ghost'' of the quasi-smooth derived scheme accessible in G-theory of the underlying scheme.  
  


\subsection{Notation}
\label{notation}
We let $\widebar X := t_0(X)$ and let $\iota_X: \widebar X \to X$ denote the natural adjunction morphism.  Further, we let $\widebar f := t_0(f)$.  Given a scheme, we let $A(X)$ and $\mr{CH}(X)$ denote the associated Chow group. 

We will be working with many composites of linear transformations.  To simplify notation we will denote composition as $AB$ or $A \cdot B$. 

\section{Derived Algebraic Geometry}
\label{derived_intro}
This section will be a (very) brief introduction to derived algebraic geometry with the main goal to set notation and basic concepts that will be used throughout the rest of the paper.  

The central goal of derived algebraic geometry is to extend the basic building blocks that make up modern algebraic geometry.  Primary extensions arise by embedding commutative rings into simplicial rings, connective $E_\infty$-ring spectra, or commutative dg-algebras.   Although straightforward in concept, and certainly one that the likes of Grothendieck, et al. could envision, one immediately encounters the necessity of dealing with the phenomena of ``homotopy coherence''.  Techniques for effectively tracking homotopy coherence while forming a reasonable theory have only recently become available (a path started by Quillen while attending Hartshorne's Duality lectures).

A key advantage (from an algebro-geometric standpoint) of working with these extended building blocks is that this is the natural context to work with the cotangent complex.  Many of the shortcomings of the cotangent complex from a modern algebro-geometric standpoint arise from the lack of fiber products to be homotopy fiber products in this extended context.  Further, derived schemes afford a greater control over the cotangent complex.  This paper will exploit this control to give Grothendieck-Riemann-Roch formulas for virtually smooth schemes. 

 As mentioned above, the various flavors of derived algebraic geometry vary in the types of affine schemes.  The results in this paper will be in the simplicial commutative world.  This is primarily for simplicity in presentation and accessibility of the paper.   We are confident that the results can be transported to all contexts.  Extra care must be taken in the $E_\infty$-context due to a more complex $Sym$ functor.  If found useful, this will be explicitly expanded on in future work.   We will refer to ``underived'' schemes as ``discrete'' since their functor of points takes values in discrete spaces, not to be confused with a scheme for which every point is open and closed.  We let $\widebar X := t_0(X)$ and let $\iota_X: \widebar X \to X$ denote the natural adjunction morphism.  This morphism is proper, an important fact derived Chow theory.

Given a derived scheme $X$, we let $\sO_X$-mod denote the stable $\infty$-category of $\sO_X$-modules.  In the case that $X$ is discrete, the homotopy category of $\sO_X$-mod is traditionally notated by $D(\sO_X\textrm{-mod})$.  We let $\QC(X)$ denote the full stable subcategory of quasi-coherent $\sO_X$-modules.

\subsection{Quasi-smooth schemes}

A quasi-smooth morphism is a morphism for which the associated relative cotangent complex is perfect and Tor-amplitude of $[-1, 0]$ (cohomologically indexed).  This class encompasses smooth morphisms and locally complete intersection morphisms.  It is stable under under (homotopy) fiber products and naturally encompasses perfect obstruction theories \cite{bf} (see \cite{timo_thesis} for more details).   Quasi-smooth embeddings arise in our theory in two ways: first, quasi-smooth embeddings will replace regular embeddings in our definition of operational Chow theory; second, quasi-smooth morphisms will be the natural class to associate the Grothendieck-Riemann-Roch formulas.  


\begin{definition}
  Given a locally free sheaf $\sE$ on a derived scheme $X$ and a section $s: \sO_X \to \sE$,  $\sO_X$ is a $Sym\; \sE^\vee$-algebra via the $0$-section and $s^\vee$ (the dual to the section $s$). Let $K(\sE; s)$ denote the explicit representative of $\sO_X \otimes_{Sym \sE^\vee} \sO_X$ obtained by taking the Bar resolution of $\sO_X$ (as the $0$-section).  This is the simplicial analog of the Koszul complex associated to ($\sE$, $s$) and has similar properties.
\end{definition}
The following local picture will be key to obtaining the main results in this paper.  It is an expansion of remarks in \cite{lurie_thesis}.
\begin{proposition}
\label{localnature}
  Let $f: X \to Y$ be a quasi-smooth embedding with $Y$ ``discrete''.  Then for any $p \in X$, there exists a Zariski neighborhood $U$ of $p$ in $Y$ such that $U \times_Y X \cong \spec K(\sE; s)$ for some locally free sheaf $\sE$ on $U$ and $s \in \sE(U)$.  In particular, $f_\ast \sO_X$ is a perfect $\sO_Y$-module. 
\end{proposition}
\begin{proof}
 Since all of these properties are local, we can assume $Y \cong \spec A$ for some discrete commutative algebra $A$,  and $B$ is a simplicial $A$-algebra with $A \to B$ inducing $A \twoheadrightarrow \pi_0(B)$.  By \cite[Proposition  3.2.16]{lurie_thesis}, if $K := \mr{fib}(f)$, then there is a natural map $q: K\otimes_A B \to \mL_{B/A}[-1]$ which is at least $2$-connective.  Since $f$ is quasi-smooth, $\mL_{B/A}[-1]$ is a projective $B$-module.  
We have a fiber/cofiber diagram
\begin{displaymath}
  \xymatrix{K^\prime \ar[r] \ar[d] & K\otimes_A B \ar[d] \\ 0 \ar[r] & \ar@{-->}@/_/[u]_{\mu} \mL_{B/A}}
\end{displaymath}
The connectedness of $q$ means $K^\prime \cong K^{\prime \prime}[1]$ for some other $B$-module; the projectivity of $\mL_{B/A}[-1]$ then ensures we a natural lift $\mu$ with $q \circ \mu$ homotopic to the identity.

If we refine our cover even more, we can assume that $\mL_{B/A}$ is free.  In other words, $\mL_{B/A} \cong (\oplus^n A) \otimes_A B[1]$.  Using standard adjointness, $\mu$ can be considered as $\mu \in \mr{Hom}_{A}(\oplus^n A,  K\otimes_A B)$.  However, the surjectivity of $\epsilon: \pi_0(K) \to \pi_0(K\otimes_AB)$ (arising from the natural unit of adjunction) combined with the fact that $A$ is discrete means there exists a lift of $\mu$: $\widetilde \mu \in \mr{Hom}(\oplus^n A, K)$ with $\widetilde \mu \circ \epsilon$ homotopic to $\mu$.  

The significance of $\widetilde \mu$ is when composed with the natural morphism $K \to A$, we obtain $s: \oplus^n A \to A$ and the algebra $\widetilde B := K((\oplus^n A)^\vee, s^\vee)$.  Since $\widetilde B \cong A \otimes^{\mL}_{A[x_1, \ldots, x_n]} A$, the morphism $A \to B$ induces $g: \widetilde B \to B$ by universality of cofiber products (this morphism sends $x_i \to 0$).  We claim $g$ is an equivalence. At the very least, it is $0$-connected.

Using the natural long exact sequence associated to the cotangent complex, one has
\begin{displaymath}
 \mL_{\widetilde B/A}\otimes B \to \mL_{B/A} \to \mL_{B/\widetilde B}
\end{displaymath}
By definition, the middle term is $\oplus^n B[1]$; by construction, the left term is $\oplus^n B[1]$.  Applying \cite[Proposition 3.2.16]{lurie_thesis} again, we know that $\mL_{B/ \widetilde B}$ is at least 2-connective, which implies that the middle morphism is surjective on $\pi_1$.  This is enough to show that $\mL_{B/\widetilde B}$ is zero (since the middle left morphism must be an equivalence then). The final equivalence between $B$ and $\widetilde B$ is provided by \cite[Corollary 3.2.17]{lurie_thesis}.

The last remark on the perfection of $f_\ast \sO_Y$ is clear from the local description.

\end{proof}

\section{Relative perfection}

In this section we will extend the theory of relatively perfect to an arbitrary morphism between derived schemes and list some of its important properties.  Many of these proofs are ``derived'' adaptations of the methods in \cite{SGA6}. We recall the following definition from \cite{Lowrey_moduli}.
\begin{definition}
  Given a morphism $f: X \to Y$, let $\bderive_{fl, f}(X)$ denote the full subcategory of $\sO_X$-mod consisting of pseudo-coherent objects of finite Tor-amplitude over $f$.
\end{definition}

\begin{definition}
\label{relperf}
  Let $f: X \to Y$ be a 
 morphism between $X, Y \in \dSch$.  Then $\sF \in \sO_X$-mod is strict $f$-relatively perfect if there exists a commutative diagram 
  \begin{displaymath}
    \xymatrix{ X \ar[rr]^i \ar[dr]^f & & Z \ar[ld]^s \\ & Y &}
  \end{displaymath}
with $s$ smooth and $i$ a closed immersion, such that $i_\ast \sF$ is in $\bderive_{fl, s}(Z)$.  $\sF$ is relatively perfect (or $f$-perfect) if there exists an \`etale cover ${U_\alpha}$ with $\sF|_{U_\alpha}$ strict relatively perfect.  The full subcategory of $\sO_X$-mod consisting of $f$-perfect objects will be denoted $\perf(f)$.
\end{definition}

Note that if $f$ is not locally of finite type over $Y$, then there will not exist a covering satisfying the factorization condition, and $\perf(f) \cong \emptyset$.

\begin{remark}
One can think of this as a stackification (in the etale topology) of the presheaf that assigns to $U \to X$, the subcategory of strict relatively perfect objects.
\end{remark}

\begin{remark}
This definition of relatively perfect will not always coincide with the notions introduced in \cite{Lowrey_moduli}.  However, if $X$ and $Y$ satisfy some strong finiteness conditions, they will align.  In particular, if $f$ is smooth and quasi-compact and $Y$ is quasi-compact,  $\perf(f) \cong \bderive_{fl, f}(X)$.   
\end{remark}

\begin{proposition}
\label{invariance}
If $\sF \in \perf(f)$, then for any local factorization of $f$, $X^\prime  \xrightarrow{i} Z^{\prime} \xrightarrow{s} Y$ as above, with $p: X^\prime \to X$ etale, 
 $i_\ast p^\ast \sF \in \perf(s)$.
\end{proposition}
\begin{proof}
  We show that for any local factorization, $i_\ast \sF$ is pseudo-coherent and locally finite relative Tor-amplitude over $Y$.  Locally finite relative Tor-amplitude is tested on $X$ and is independent of the factorization (due to the projection formula).  We are thus reduced to showing pseudo-coherence.  Let $i^{\prime \prime}: X^{\prime \prime} \to Z^{\prime \prime}$ and $s^{\prime \prime}: Z^{\prime \prime} \to Y$ be a local factorization with $X^{\prime \prime}$ an \'etale cover of $X$ such that $i^{\prime \prime}_\ast \sF$ satisfies the desired properties (we know at least one exists since $\sF \in \perf(f)$ by assumption).  This yields the diagram

\begin{displaymath}
  \xymatrix@=1em{&  X^\prime \times_X X^{\prime \prime}\ar[rrrr]^{k := (i^\prime \times i^\second)} \ar[rdd]^(.4){p^{\prime \prime}} \ar[ld]_(.6){p^{\prime}} && & & Z^\prime \times_Y Z^{\prime \prime} \ar[ddr] \ar[ld] &  \\
X^\prime \ar[ddr]^p \ar[rrrr]_(.3){i^\prime}|(.5)\hole & &&&Z^\prime \ar[rdd]|(.5)\hole&&\\
& &  X^{\prime \prime}\ar[ld] \ar[rrrr]_(.4){i^{\prime \prime}} & &  & & Z^{\prime \prime}\ar[ld] \\
& X  \ar[rrrr] &  &&& Y
}
\end{displaymath}
The left vertical morphisms are \'etale and the right vertical morphism are smooth.  All but the bottom horizontal morphism are closed immersions.  The proposition will be shown if we show $(i^{\prime} \times i^{\prime \prime})_\ast p^{\prime \ast} p^{\ast}\sF$ satisfies the desired properties if and only if $i^\prime_\ast p^\ast \sF$ does.  This will proceed in two steps.  We first assume that $p^{\prime}$ is a Zariski open immersion (i.e., $X^{\prime \prime}$ is a Zariski cover of $X$), and then generalize to etale covers.

Since the little Zariski site over a derived scheme is equivalent to the little Zariski over its discrete truncation, and in this case the Zariski subspace topology (as a closed subvariety of $Z^\prime$) and the Zariski site agree, any Zariski open neighborhood of $X^\prime$ is the restriction of a Zariski open neighborhood of $Z^\prime$.  We can thus form the Cartesian diagram 
\begin{displaymath}
  \xymatrix{X^\prime \times_X X^{\prime \prime}  \ar[r]^{\quad j} \ar[d]^{p^\prime} & W \ar[d]^q \\ X^\prime \ar[r]^{i^\prime} & Z^\prime}
\end{displaymath}
with $q$ a Zariski open immersion and $j$ a closed immersion.  By Lemma \ref{ifandonlyif}, $i^\prime_\ast \sF$ is pseudo-coherent if and only if $j_\ast p^{\prime \ast} \sF$ is as well; from Lemma \ref{closedequiv} $j_\ast p^{\prime \ast} \sF$ is pseudo-coherent if and only if $k_\ast p^{\prime \ast} \sF$ is.  Putting this together gives the result for Zariski.

For $p^\prime$ etale, one needs an extra step since one can only Zariski locally extend etale morphisms on closed embeddings.  In particular, it is well known that there exists a Zariski cover $\coprod V_\alpha  \to X^\prime \times_X X^{\prime \prime}$ such that $V_\alpha \to X^\prime$ is the pullback of an etale morphism $U_\alpha  \to Z^\prime$.  One then has a diagram
\begin{displaymath}
  \xymatrix{\coprod V_\alpha \ar[r] \ar[d] & W \ar[d] \\ X^\prime \times_X X^{\prime \prime} \ar[r] & Z^\prime \times_Y Z^{\prime \prime} }
\end{displaymath}
with the vertical morphisms Zariski open.  Using the same arguments as above shows that I can assume that $X^\prime \times_X X^{\prime \prime}$ is the restriction of an etale $W^\prime \to Z^\prime$.  The above arguments proceed the same to give the result.
\end{proof}

\begin{corollary}
  If $f = s \circ i$ is a factorization with $s$ smooth and $i$ and
  embedding, then $\sF \in \perf(f)$ if an only if it is strict
  $f$-relatively perfect.
\end{corollary}

\begin{lemma}
\label{smoothperfect}
Given a faithfully flat smooth morphism $h: Z \to Y$, $\sH \in \QC(Y)$ is perfect if and only if $h^\ast \sH$ is.
\end{lemma}
\begin{proof}
  One direction is clear.  From Lemma \ref{nicepush}, we know that $h^\ast \sH$ is pseudo-coherent if and only if $\sH$ is.  One then just needs to check Tor-amplitude.  This is clear from the faithfullness of $h$ and the fact that tensoring commutes with pullback.   
\end{proof}

\begin{remark}
  Up until now, we have only used the $fppf$-part of smoothness.  However, the next result uses smoothness extensively.  This lemma will be very important for understanding the map between bivariant theories.
\end{remark}

\begin{lemma}
\label{globallyperfect}
Let $\sF \in \QC(X)$ and $f = s \circ i$ any factorization, then $\sF \in \perf(f)$ if and only if $i_\ast \sF \in \perf(Z)$.
\end{lemma}
\begin{proof}
By Proposition \ref{invariance}, $\sF \in \perf(f)$ if and only if $i_\ast \sF \in \perf(s)$, thus we can assume that $i$ is the identity and $f = s$.  The locality of the result implies we must show that for $\sF$ strict pseudo-coherent,  $\sF$ is finite Tor-amplitude if and only if it is relatively finite Tor-amplitude over the smooth morphism $f$.  One implication is clear, for the other, the flatness of $f$ and the fiber diagram 
\begin{displaymath}
  \xymatrix{X \times_Y X \ar[r]^{\pi_1} \ar[d]^{\pi_2} & X \ar[d]^f \\ X \ar[r]^f & Y}
\end{displaymath}
ensures $\pi_2^\ast \sF$ is relatively perfect over $\pi_1$.  The smoothness of $f$  implies the natural diagonal $\Delta: X \to X \times_Y X$ is a quasi-smooth embedding.   We claim $\sG := \Delta_\ast \Delta^\ast \pi_2^\ast\sF \in \perf(\pi_1 \circ \Delta)$.  Since $\pi_1$ is smooth, for $\sG \in  \perf(\pi_1 \circ \Delta)$, it is enough for $\Delta_\ast \sG$ to be pseudo-coherent and locally finite Tor-amplitude over $\pi_1$.   However, $\Delta_\ast \sG = \Delta_\ast \Delta^\ast \pi_2^\ast \sG$ and locally is explicitly given as tensoring $\pi_2^\ast \sF$ with the Koszul resolution of $\Delta_\ast(\sO_X)$.  This operation preserves pseudo-coherence and finite Tor-amplitude, thus $\sG \in \perf(\pi_1 \circ \Delta) \cong \perf(id)$ and the result is shown.  
\end{proof}

\section{Bivariant algebraic K-theory and bivariant Chow}

In this section we define the two bivariant theories used in the Grothendieck-Riemann-Roch formula.  Throughout the rest of the paper (not including appendices) $\dsch_{/B}$ will consist of quasi-projective (over $B$) derived schemes with $\widebar X$ locally of finite type over a smooth base $B$.
\begin{remark}
  The strong conditions on the base are necessary for push-forward on
  the K-theory side, a well behaved Chow theory, and the comparison of $\tau$ (defined below) with
  the product. There might be alternate approaches for more general bases.  At a minimum it would require a more complex Chow functor.
\end{remark}

The following lemmas are immediate consequences of these assumptions.
\begin{lemma}
\label{factorizeit}
  Let $f \in \mr{Hom}_{\dsch_{/B}}(X, Y)$.  Then there exists a factorization of $f = s\circ i$ where $s: S \to Y$ is smooth and $i: X \to S$ is a closed embedding. 
\end{lemma}
\begin{proof}

The existence of the factorization follows from $t_0(X)$ being quasi-projective: we can find a closed embedding $j$ into a Zariski open subset of projective space over $B$.  Let $M$ be this open subset.  Then for any $Y \to B$, $Y \times_{B} M \to Y$ is a smooth morphism, and $i = f \times j:  X \to Y \times M$ is a closed embedding.
\end{proof}

For the next lemma, the reader is referred to \cite{BZ} for the definition of a perfect morphism.
\begin{lemma}
\label{everythingsperfect}
  All morphisms in $\dsch_{/B}$ are perfect, and therefore satisfy base change and the projection formula.
\end{lemma}
\begin{proof}
  This is an easy consequence of the fact that a derived scheme $X$ with $\widebar X$ quasi-projective is a perfect derived scheme since it is locally quasi-compact, quasi-separated, and has an ample sequence.  Base change and projection formula follows then from \cite[Section 3.2]{BZ}. 
\end{proof}
These two lemmas will simplify the theory considerably.  

\subsection{Bivariant algebraic K-theory}

For bivariant algebraic K-theory, the confined morphisms are proper morphisms and the ``independent squares'' are homotopy fiber diagrams.  Given $[f] \in \pi_0(\map_{\dsch_{/B}}(X, Y))$, we set $K([f]) = K_0(\perf(f))$ where $f \in \map_{\dsch_{/B}}(X, Y)$ is any representative of the class $[f]$.  We show below that $K([f])$ is indeed independent of the choice of representative.  We will now define the various bivariant operations. 

Given a composable pair of morphisms $[f], [g]$,
\begin{align*}
  K([f]) \otimes K([g]) &\to K([f \circ g])\\
  [\sF] \otimes [\sG] & \to [g^\ast \sG \otimes_{\sO_X} \sF]
\end{align*}
with $\sG \in \perf(g)$ and $\sF \in \perf(f)$ (obtained by fixing a representative for each class of morphisms).  The proposition below ensures $g^\ast \sG \otimes_{\sO_X} \sF \in \perf(g \circ f)$.  Since this formula uses two exact transformations of stable $\infty$-categories it will descend to the Grothendieck groups and be independent of the choice of representatives.  

\begin{lemma}
\label{products_work}
If $f: X \to Y$,  $g: Y \to Z$, $\sF \in \perf(f)$ and $\sG \in \perf(g)$, then $f^\ast \sG \otimes_{\sO_X} \sF \in \perf(g \circ f)$   
\end{lemma}
\begin{proof}
  The question is local on $X$, so, similar to \cite[Proposition 1.11]{SGA6}, we replace $X$ with an open neighborhood.  We can further refine to assume that $\widebar{f}$, $\widebar g$ and $\widebar{g \circ f}$ are finite type. We then have the following diagram:
  \begin{displaymath}
    \xymatrix{X \ar[r]^i \ar[rd]^f & \mA^a_Y \ar[d]^{s^\prime} \ar[r]^j & \mA^{a + b}_Z \ar[d]^{s} \\ & Y \ar[r]^k \ar[rd]^g & \mA_Z^b \ar[d] \\ & & Z}
  \end{displaymath}
The horizontal morphisms are closed immersions and the vertical morphisms are smooth.   From Lemma \ref{globallyperfect}, $k_\ast \sG$ is perfect. By Lemma \ref{smoothperfect}, $k_\ast \sG$ is perfect if and only if $s^{\ast} k_\ast \sG$ is as well.  Applying base change, we see that $\sG \in \perf(g)$ if and only if $j_\ast s^{\prime \ast} \sG \in \perf(\mA^{a + b}_Z)$.  By Lemma \ref{invariance}, this is true if and only if $s^{\prime\ast} \sG \in \perf(j)$.

The goal is to show that $j_\ast i_\ast (\sF \otimes_{\sO_X} f^\ast \sG)$ is perfect. Since $f^\ast \cong i^\ast s^{\prime \ast}$, after applying the projection formula for $i$, the above observation implies that we can start with $p^\ast \sG$ and replace $f$ and $g$ with $i$ and $j$, i.e., we can assume that $f$ and $g$ are closed immersions.  With the setup, the result is now clear: by the projection formula, $(g \circ f)_\ast (\sF \otimes f^\ast \sG) \cong g_\ast (f_\ast \sF \otimes \sG)$.   By definition $f_\ast \sF$ is perfect on $Y$.  Reducing our locality again, we can assume that $f_\ast \sF$ is strict perfect and is  ``built'' out of a finite number of locally free $\sO_Y$-modules.  The problem is thus reduced to the case that $f_\ast \sF$ is locally free.  Reducing locality once again, we can assume that $f_\ast \sF$ is free and the result is clear.  
\end{proof}

The pushforward operation $K([f \circ g]) \to K([g])$ for $f$ a proper morphism is defined as $[f_\ast]([\sF]) := [f_\ast \sF]$ for any $\sF \in \perf(f \circ g)$.  The following theorem, proved exactly as  \cite[Proposition 4.8]{SGA6}, ensures the validity of the operation.

\begin{proposition}
\label{bivariantpush}
 Given the commutative diagram
 \begin{displaymath}
   \xymatrix{ X \ar[rr]^j \ar[dr]^f & & W \ar[ld]^h \\ & Y &}
 \end{displaymath}
with $j$ proper then $j_\ast \sF \in \perf(h)$ for any $\sF \in \perf(f)$.
\end{proposition}


For the last operation, given $Y^\prime \xrightarrow{h} Y$ and $X \xrightarrow{f} Y$, we obtain a homotopy fiber diagram
 \begin{displaymath}
   \xymatrix{W \ar[r]^{f^\prime} \ar[d]^{h^\prime} & V \ar[d]^h \\ X \ar[r]^f & Y}
 \end{displaymath}
we define $K(h) \to K(h^\prime)$ as $[f^\ast]([\sF]) := [f^{\prime *} \sF]$.  We must show that $f^{\prime \ast} \sF \in \perf(g^\prime)$ for any $\sF \in \perf(g)$.  Using Lemma \ref{factorizeit}, factorize $g$ and pullback the factorization to get a series of homotopy fiber squares: 
\begin{align*}
\xymatrix{& X \times_Y S \ar@/_1pc/[dd]^(.3){s^\prime}|\hole \ar[r]^{f^{\prime \prime}} & S \ar@/_1pc/[dd]^{s} \\ X \times_Y Y^\prime \ar[r]^{f^\prime} \ar[rd]^{g^\prime} \ar[ru]^{i^\prime} &  Y^\prime \ar[rd]^g \ar[ru]^i & \\&  X \ar[r]^f &  Y} 
\end{align*}
Stability of closed and smooth morphisms under base change shows $X \times_Y Y^\prime \xrightarrow{i^\prime} X\times_Y S \xrightarrow{s^\prime} X \times_Y Y = X$ is a global factorization of $g^\prime$.  Since $i$ is proper, base change implies that $i^\prime_{\ast} f^{\prime \ast} \sF \cong f^{\prime \prime \ast} i_\ast \sF$.  However, by \ref{globallyperfect} $i_\ast \sF$ is perfect.  Since perfect is stable under pullback, the result then follows.




\begin{lemma}
\label{sameoverafield}
  For $f: X \to Y$ with $Y$ smooth over a field $k$, then $K([f]) \cong K([t_0(f)])$.
\end{lemma}
\begin{proof}
 Since $\iota_X$ is proper there is a natural morphism $K([t_0(f)]) \xrightarrow{\iota_{X\ast}} K([f])$.    Let $g: Y \to \spec k$ be the structure morphism and $X \xrightarrow{i} S \xrightarrow{s} Y$ a choice of global factorization. The global factorization chosen for $f$ will suffice for $t_0(f)$ as well.  Composition of smooth morphisms is smooth, so $S$ is smooth over $k$,  and $\mr{D}^b_{fl, g \circ s}(S) \cong \perf(S)$.  Since $S$ is discrete and $fppf$ over $\spec k$, it is locally Noetherian and $\mr{D}^b_{fl, g \circ j}(S) \cong \bderive(S)$, which is closed under the natural t-structure truncations $\tau_{\leq n}$.  

This is enough to show for $\sF \in \perf(f)$, $\sH^i(\sF) \in \perf(f)$.  Since $\sH^i(\sF)$ is naturally a $\sO_{\widebar X}$-module, $\sH^i(\sF) \in \perf(t_0(f))$.  The Harder-Narasimhan filtration provided by the t-structure ensures the existence of a $\alpha \in K([t_0(f)])$ such that $\iota_{X\ast}\alpha = [\sF]$.  Thus $[\iota_{X \ast}]$ is surjective.  Injectivity follows from the fact that if $\sH \to \sG \to \sI$ is a distinguished triangle in $\perf(f)$, then the resulting relation on the replacements $\Sigma [\sH^i(\sH)]$, $\Sigma [\sH^i(\sG)]$, and $\Sigma [\sH^i(\sI)]$ can be deduced from the accompanying long exact sequence of cohomology groups (induced by the t-structure).  This latter object exists in the realm of $\perf(f \circ \iota_X)$. Since $\iota_{X}$ is conservative, injectivity follows.
\end{proof}
\begin{remark}
  Since $B$ is assumed to be smooth, this result extends to the case that $Y$ is smooth over the base.
\end{remark}

\begin{question}
  It is currently unclear if $K(f) = K(t_0(f))$ for all $f$.
\end{question}

\subsubsection{Verification of bivariant identities}

The verification of A1) associativity of the product,  A2) functorality of pushforward,   A3) functorality of pullback,  A13) commutativity of pullback with product,  all follow directly from the definitions and well known properties of pushforward ($f_\ast$) and pullback ($f^\ast$).  Property A12) commutativity of pushforward with product is an easy consequence of the projection formula (which applies after Lemma \ref{everythingsperfect});  property A23) commutativity of pushforward and pullback follows from base change.  We will explicitly work out property A123) satisfies the projection formula.

Given the diagram

\begin{displaymath}
  \xymatrix{X^\prime \ar[r]^{g^\prime} \ar[d]^{f^\prime} & X \ar[d]^f &  \\ Y^\prime \ar[r]^g & Y \ar[r]^h & Z}
\end{displaymath}
with $g$ proper.  Let $\sF \in \perf(f)$ and  $\sG \in \perf(h \circ g)$, then
\begin{align*}
 [g^\prime_\ast](([g^\ast][\sF]) \cdot [\sG]) &=  [g^\prime_\ast (g^{\prime \ast} \sF \otimes_{\sO_X} f^{\prime \ast} \sG)]\\ 
&=  [\sF \otimes g^\prime_\ast f^{\prime \ast}\sG] \qquad \textrm{from the projection formula} \\
&= [\sG \otimes f^\ast g_\ast \sG] \qquad \textrm{from base change}\\
& = ([g_\ast][\sG]) \cdot [\sF]
\end{align*}

\subsubsection{The orientation}

\begin{lemma}
\label{qsorientable}
Suppose $f: X \to Y$ is quasi-smooth.  Then $\sO_X$ is $f$-perfect.
\end{lemma} 
\begin{proof}
Using methods similar to Proposition \ref{localnature}, locally on $X$ and $Y$ we can factor $f$ as a quasi-smooth embedding into affine space over an open subset of $Y$, followed by the natural projection (see the discussion preceding \cite[Proposition 3.4.17]{lurie_thesis}).  The result then follows from Proposition \ref{invariance} and Proposition \ref{localnature}.


\end{proof}

This lemma allows us to define an orientation to the class of quasi-smooth morphisms: let $\theta$ be the orientation of the bivariant theory $K$ determined by $\theta(f) = [\sO_X] \in K([f])$ for $f: X \to Y$ quasi-smooth.  This orientation satisfies the property $\theta(f) \cdot \theta(g) = \theta(fg)$.  We say that this orientation is ``multiplicative''.

\subsection{Bivariant operational derived Chow}

\subsubsection{$\dCh(X)$}
Given $X \in \dSch_{B}$, define $\dCh(X)$ via the standard definition: the free abelian generated by algebraic subvarieties modulo rational equivalence.   Here by algebraic variety we mean discrete algebraic variety.  This ensures that the standard formulas for pushforward and flat pullback remain valid.  It is clear that $\iota_X$, being a proper morphism, induces an isomorphism (via the universality of the adjunction unit morphism).   

\begin{lemma}
\label{pushforward_formula}
If $f: X \to Y$ is proper,   $f_\ast = \iota_{X\ast}^{-1} \circ t_0(f)_\ast \circ \iota_{Y\ast}$.  Likewise, if $f: X \to Y$ is flat, then $f^\ast = \iota_{Y\ast}^{-1} \circ t_0(f)^{-1} \circ \iota_{X\ast}$
\end{lemma}
\begin{proof}
  Let $\alpha \in \dCh(X)$.  By linearity, we can assume $\alpha = [V]$ for some algebraic variety $V \xrightarrow{h} X$.  By universality, $h = \iota_X \circ h^\prime$ and we have a commutative diagram
  \begin{displaymath}
    \xymatrix{ V \ar[d]^{h^\prime}  & & \\ \widebar{X} \ar[r]^{t_0(f)} \ar[d]^{\iota_X} & \widebar{Y} \ar[d]^{\iota_Y} \\ X \ar[r]^f & Y}
  \end{displaymath}
Thus, $f \circ h = \iota_Y \circ t_0(f) \circ i^\prime$.  The formula then follows.  The statements for $f$ flat are similar, one only needs to observe that $t_0(f)$ is flat by flat base change since $\widebar X \cong X \times_Y \widebar Y$ in this case.
\end{proof}

The following lemma is similar to the ``underived'' version and is stated without proof.
\begin{lemma}
\label{base_change}
  Given the commutative diagram
  \begin{displaymath}
     \xymatrix{W \ar[r]^{f^\prime} \ar[d]^{g\prime} & V \ar[d]^g \\ X \ar[r]^f & Y}
  \end{displaymath}
with $f$ proper and $g$ flat, then $g^\ast f_\ast = f_\ast^{\prime}g^{\prime \ast}$.
\end{lemma}

We now define virtual Gysin homomorphisms for quasi-smooth embeddings.  The functorality of the following construction can be found in \cite{virt_pullback}.  If $j: X \to Y$ is a quasi-smooth embedding, then $\iota_X^\ast \mL_{X/Y}$ is locally free sheaf.  We define $j^!_{vir}$ as the series of morphisms
\begin{align}
\label{virtual_def}
  \dCh(Y) \xrightarrow{\iota_Y^{-1}} \dCh(\widebar Y) \xrightarrow \dCh(C_{\widebar X} \widebar Y) \rightarrow \mV(\iota_X^\ast \mL_{X/Y}) \xrightarrow{\pi^{\ast -1}} \dCh(\widebar X) \xrightarrow{\iota_{X\ast}} \dCh(X)
\end{align}

\subsubsection{The bivariant operational theory}

The confined morphisms and ``independent squares'' for derived Chow are the same as K-theory.  Namely, let $\dCh$ denote the operational theory with the confined maps consisting of proper maps and the ``independent squares'' consisting of  homotopy fiber squares.  

The operational adjective means given $[f] \in \pi_0(\hom_{\dSch_S}(X, Y))$, $\sigma \in \dCh([f])^j$ consists of the following data:  for each $V \to Y$ and homotopy fiber square
\begin{displaymath}
  \xymatrix{W \ar[r] \ar[d] & V \ar[d] \\ X \ar[r] & Y}
\end{displaymath}
a $j$-graded homomorphism $\sigma^V_W: \dCh(V) \to \dCh(W)$ compatible with proper pushforward, flat pullback, and virtual Gysin homomorphisms induced by quasi-smooth embeddings.  
\begin{notation}
We denote the product of two operational classes $\sigma$, $\omega$ as $\sigma \cdot \omega$.  When we are referring to  pullback, pushforward, and virtual Gysin homomorphisms as bivariant operations, we will use brackets.  To provide clarity we will sometimes put the source and target groups as superscripts and subscripts, respectively.  
\end{notation}

\begin{proposition}
\label{gysin_commute}
Suppose $\sigma$ commutes with proper pushforward and flat pullback.  Then $\sigma \in \dCh([f])$  commutes with virtual Gysin homomorphisms if and only if it commutes with the case of $\{0\} \to \mA_k^1$, where $k$ is our base field.
\end{proposition}
\begin{proof}
Since $\{0\} \to \mA_k^1$ is quasi-smooth embedding, one direction is clear.  Suppose $\sigma$ commutes with this embedding and let $Z^\prime \xrightarrow{j} Z$ be a quasi-smooth embedding.   This embedding gives rise to the following diagram, with all squares Cartesian. 
\begin{displaymath}
\xymatrix{W^\prime \ar[r]^{l^\prime} \ar[d]^{j^{\prime \prime}} & V^\prime \ar[r]^{h^\prime} \ar[d]^{j^\prime} &  Z^\prime \ar[d]^{j} \\ W \ar[r]^l \ar[d] & V \ar[d]^g \ar[r]^h & Z \\ X \ar[r]^f & Y &   }
\end{displaymath}
We must show $j^{\second !}_{vir} \cdot \sigma^V_W = \sigma^{V^\prime}_{W^\prime} \cdot j^{\prime !}_{vir}$.   To do so, we will obtain a different formulation for $j^{\prime !}_{vir}$. 

To simplify notation, let $C^{V^\prime}_{\widebar{Z^\prime}} \widebar Z \cong V^\prime \times_{\widebar{Z^\prime}} C_{\widebar{Z^\prime}} \widebar Z$ and  $M^{V}_{\widebar{Z^\prime}} \widebar Z \cong V \times_Z M^\circ_{\widebar{Z^\prime}} \widebar Z$ (likewise for $W^\prime$ and $W$, respectively).  Further, denote $N = \mV(\mL_{Z^\prime/Z})$ and $N_{V^\prime} = V^\prime \times_{Z^\prime} N$ (likewise for $W^\prime$) and $\pi_{Z^\prime}: N \to Z^\prime$ (likewise for $V^\prime$ and $W^\prime$) be the natural morphism.    Note that  $V \times_Z C_{\widebar{Z^\prime}} \widebar Z \cong V^\prime \times_{Z^\prime} C_{\widebar{Z^\prime}} \widebar Z$
 and the flatness of $\pi_{Z^\prime}$ means $t_0(N) \cong \mV(\iota_{Z^\prime}^\ast \mL_{Z^\prime/Z})$. Since $h^{\prime \ast} \mL_{Z^\prime/Z} \cong \mL_{V^\prime/V}$, $N_{V^\prime} \cong \mV(\mL_{V^\prime/V})$ and $\mV(\iota_{V^\prime}^* \mL_{V^\prime/V}) \cong \widebar{N_{V^\prime}}$.  

The morphism $C_{\widebar{V^\prime}} \widebar V \to \widebar{N_{V^\prime}}$ in the definition of of $j^{\prime !}_{vir}$ (equation (\ref{virtual_def}) above) factors as
\begin{displaymath}
C_{\widebar{V^\prime}} \widebar V  \hookrightarrow \overline{C^{V^\prime}_{\widebar{Z^\prime}} \widebar{Z}} \hookrightarrow \widebar{N_{V^\prime}}
\end{displaymath}
Locally this follows from the following fact: if  $\widebar Z = \spec A$, $\widebar{Z^\prime} = \spec A/I$, $\widebar V = \spec B$, and $\widebar{V^\prime} = \spec B/J$ then then there is a natural surjective map from $B\otimes I \to J$ (underived).  Similarly, there exists a natural morphism $M^\circ_{\widebar{V^\prime}} \widebar V \to \overline{M^V_{\widebar{Z^\prime}} \widebar Z}$.
We will show the composite
\begin{displaymath}
\eta_V: \dCh(V) \to \dCh(C_{\widebar{V^\prime}} \widebar V) \to \dCh(\overline{C^{V^\prime}_{\widebar{Z^\prime}} \widebar Z}) \xrightarrow{\iota_{C^{V^\prime}_{\widebar{Z^\prime}} \widebar Z\ast}} \dCh(C^{V^\prime}_{\widebar{Z^\prime}} \widebar Z)
\end{displaymath}
can be alternatively defined (see below) and with this alternative formulation it will be clear $\sigma^{C^{V^\prime}_{\widebar{Z^\prime}}\widebar Z}_{C^{ W^\prime}_{\widebar{Z^\prime}}\widebar Z} \cdot \eta_V =  \eta_W \cdot  \sigma^V_W$ (where $\eta_W$ is defined similarly). 


The embedding of the normal cone gives rise to the diagram with each square Cartesian.  
\begin{displaymath}
  \xymatrix{C_{\widebar{Z^\prime}}^{W^\prime} \widebar Z \ar[d]^{i_0^{\prime \prime}} \ar[r] & C_{\widebar{Z^\prime}}^{V^\prime} \widebar Z \ar[r] \ar[d]^{i_0^\prime} & C_{\widebar{Z^\prime}} \widebar Z \ar[d]^{i_0} \ar[r] & 0 \ar[d] \\ M_{\widebar{Z^\prime}}^{ W} \widebar Z \ar[r] \ar[d] & M_{\widebar{Z^\prime}}^{V} \widebar Z \ar[r] \ar[d] & M^\circ_{\widebar{Z^\prime}} \widebar Z \ar[d] \ar[r] & \mA^1 \\ W \ar[r] \ar[d] & V \ar[r] \ar[d] & Z & \\ X \ar[r]^f & Y & &}
\end{displaymath}
Note that although $M^\circ_{\widebar{Z^\prime}} \widebar Z$ is discrete, $M^V_{\widebar{Z^\prime}} \widebar Z$ may not be.  The explicit formulas for proper pushforward and flat pullback ensure the localization sequence still holds for Zariski open subschemes.  The following commutative diagram is a consequence of $\sigma$ commuting with flat pullback and proper pushforward.

\begin{displaymath}
\xymatrix{ \dCh(C_{\widebar{Z^\prime}}^{W^\prime} \widebar Z) \ar[r]^{i_{0 \ast}^{\prime \prime}} \ar[d]^{\sigma^{C^{V^\prime}_{\widebar{Z^\prime}}\widebar Z}_{C^{W^\prime}_{\widebar{Z^\prime}}\widebar Z}} & \dCh(M_{\widebar{Z^\prime}}^W \widebar Z) \ar[r]^(.6){k^{\prime \prime \ast}} \ar[d]^{\sigma^{M^V_{\widebar{Z^\prime}}\widebar Z}_{M^{W}_{\widebar{Z^\prime}}\widebar Z}} & \mA^1_{W} \ar[d]^{\sigma^{\mA^1_{W}}_{\mA^1_{V}}} \ar[r]& \ar[d] 0 \\ 
\dCh(C_{\widebar{Z^\prime}}^{V^\prime} \widebar Z) \ar[r]^{i_{0 \ast}^{\prime}} \ar[r] & \dCh(M_{\widebar{Z^\prime}}^{V} \widebar Z) \ar[r]^(.6){k^{\prime \ast}} & \mA^1_{V} \ar[r] & 0}
\end{displaymath}
The morphism $(i^\prime_{0})^!_{vir}$  restricts to zero on the image of $i_{0 \ast}^\prime$ (since $\mL_{\{0\}/\mA^1}$ is a rank 1 trivial (shifted) locally free sheaf).  The exactness of the rows allows us to consider $(i_0^{\prime})^!_{vir}$ as a morphism $\dCh(\mA^1_{V}) \to \dCh(C_{\widebar{Z^\prime}}^{V^\prime} \widebar Z)$ and we have the composite 
\begin{displaymath}
\dCh(V) \xrightarrow{\mr{pr}^\ast} \dCh( \mA^1_{V}) \xrightarrow{(i^\prime_{0})^!_{vir}}  \dCh(C_{\widebar{Z^\prime}}^{V^\prime} \widebar Z)
\end{displaymath}
We claim $\eta_V = (i^\prime_0)_{vir}^! \circ \mr{pr}^\ast$.  This more or less follows from the similar case in \cite{fulton_intersection}. Let $[A] \in \dCh(V)$ where $A$ is a closed subvariety.  Then $\mr{pr}^\ast([A]) = [A \times \mA^1]$.  The natural morphism $M^\circ_{\widebar{V^\prime}} \widebar V \to M^V_{\widebar{Z^\prime}} \widebar Z$ ensures the existence of the of a class $\alpha = [M^\circ_{\widebar{V^\prime} \cap A} A] \in  \dCh(M^V_{\widebar{Z^\prime}} \widebar Z)$.  It is clear that $k^{\prime \ast} \alpha \cong \mr{pr}^\ast([A])$ and $(i_0^\prime)^!_{vir} \alpha \cong [C_{\widebar{V^\prime}\cap A} A]$.  In other words, we have the commutative diagram  
\begin{displaymath}
  \xymatrix{& \dCh(C_{\widebar{V^\prime}} \widebar V) \ar[dr] & \\ \dCh(V) \ar[ru] \ar[rr]^{(i^\prime_0)_{vir}^! \circ \mr{pr}^\ast} & & \dCh(C^{V^\prime}_{\widebar{Z^\prime}} \widebar Z)}
\end{displaymath}
The top composite is $\eta_V$, showing the claim.  Similar statements apply to  $(i_{0 \ast}^{\prime \prime})^!_{vir}$ and the definition of $\eta_W$.    By assumption $\sigma^{C^{V^\prime}_{\widebar{Z^\prime}}\widebar Z}_{C^{W^\prime}_{\widebar{Z^\prime}}\widebar Z} \cdot (i^\prime_{0})^!_{vir} = (i^\second_{0})^!_{vir}  \cdot \sigma^{M^{V}_{\widebar{Z^\prime}}\widebar Z}_{M^{W}_{\widebar{Z^\prime}}\widebar Z}$.  Therefore, $\sigma$ commutes with all morphisms defining $\eta_W$ and $\eta_V$; the desired commutativity of $\sigma$ with $\eta$ follows.

Let $m_{Z^\prime}: C_{\widebar{Z^\prime}} \widebar Z \to N$ denote the natural proper map.  We have the fiber diagram 
\begin{displaymath}
  \xymatrix{ C^{ W^\prime}_{\widebar{Z^\prime}} \widebar Z \ar[r] \ar[d]^{m_{ W^\prime}} &  C^{V^\prime}_{\widebar{Z^\prime}} \widebar Z \ar[r] \ar[d]^{m_{ V^\prime}} & C_{\widebar{Z^\prime}} \widebar Z \ar[d]^{m_{Z^\prime}}  \\
N_{W^\prime} \ar[r] \ar[d]^{\pi_{ W^\prime}} & N_{ V^\prime} \ar[r] \ar[d]^{\pi_{V^\prime}} & N \ar[d]^{\pi_{Z^\prime}}  \\ W^\prime \ar[r] \ar[d] &  V^{\prime} \ar[r] \ar[d] & Z^\prime     \\ X  \ar[r] & Y & }
\end{displaymath}

From above, it is clear
\begin{displaymath}
(j^\prime)^!_{vir} = \dCh(V) \xrightarrow{\eta_V} \dCh(C^{V^\prime}_{\widebar{Z^\prime}} \widebar Z) \xrightarrow{ (\iota_{C^{V^\prime}_{\widebar{Z^\prime}} \widebar Z \ast})^{-1}} \dCh(\overline{C^{V^\prime}_{\widebar{Z^\prime}} \widebar Z}) \xrightarrow{\widebar{m_{V^\prime}}_\ast} \widebar{N_{Z^\prime}} \xrightarrow{\iota_{V^\prime \ast} \circ (\widebar{\pi_{V^\prime}}^{\ast})^{-1}} \dCh(V^\prime) 
\end{displaymath}
It follows from Lemma \ref{pushforward_formula}, the composite of the last 3 morphisms can be alternatively defined as $(\pi_{V^{\prime}}^*)^{-1} \circ m_{V^\prime \ast}$ and thus $[j^!_{vir}]^V_{V^\prime} = j^{\prime !}_{vir}= (\pi_V^{*})^{-1} \circ m_{V\ast} \circ \eta_{V}$.  We now have
 \begin{align*}
   \sigma^{V^\prime}_{W^\prime} \cdot [j_{vir}^!]^V_{V^\prime} & = \sigma^{V^\prime}_{W^\prime} \cdot  (\pi_{V^\prime}^{\ast})^{-1} \cdot m_{V^\prime \ast} \cdot \eta_{V} \\
  & = (\pi_{W^\prime}^{*})^{-1}\cdot \sigma^{N_{V^\prime}}_{N_{W^\prime}} \cdot  m_{W^\prime \ast} \cdot \eta_{V} \quad  \textrm{(commutativity with $[\pi^\ast]$)} \\
 & = (\pi_{W}^*)^{-1} \cdot m_{W\ast} \cdot \sigma^{C^{V^\prime}_{\widebar{Z^\prime}}\widebar Z}_{C^{W^\prime}_{\widebar{Z^\prime}}\widebar Z} \cdot  \eta_{V}  \quad \textrm{(commutativity  with pushforward)}\\
 & = \pi_{W}^{*-1}\cdot m_{\widebar{W^\prime} \ast} \cdot \eta_{W} \cdot \sigma^{V}_{W}\\
 & = [j_{vir}^{!}]^{W}_{W^\prime} \cdot  \sigma^{V}_{W}
  \end{align*}
and commutativity is shown.


\end{proof}

\begin{corollary}
\label{flat_qs_ops}
  Let $f: X \to Y$.  If 
  \begin{enumerate}
  \item $f$ is flat, then $[f^\ast] \in \dCh([f])$
  \item $f$ is quasi-smooth, then $[f^!]_{vir} \in \dCh([f])$
  \item $f$ is smooth, then $[f^!]_{vir} = [f^*]$.
  \end{enumerate}
\end{corollary}

\begin{lemma}
  Given $X, Y \in \dSch_{S}$, with $\widebar X = X$ and $\widebar Y = Y$, then
  for any $f: X \to Y$, $\dCh([f]) \cong A(X \xrightarrow{f} Y)$ (in
  the notation of \cite{fulton_intersection}).
\end{lemma}
\begin{proof}
Compatibility with pushforward shows we can assume that $V$ is discrete.  In this case, $\iota_{W \ast}$, being an isomorphism shows any $\sigma \in A(X \xrightarrow{f} Y)$ can be promoted to an element the equality $\widetilde \sigma \in \dCh([f])$.  This is done by $\widetilde \sigma^{V}_W = \iota_{W \ast} \sigma^{V}_{\widebar W}$.  Conversely, to any $\nu \in \dCh([f])$, $\bar{\nu}^{V}_{\widebar W} = \iota_{W\ast}^{-1} \nu^{V}_{W}$.  One checks that these assignments commute with proper pushforward and flat pullback.  One is left showing that they commute with Gysin homomorphisms, but this is obvious after Lemma \ref{gysin_commute}.  
\end{proof}

\begin{lemma}
\label{chowsameasfield}
  For any $f: X \to Y$, $\dCh([f]) \cong \dCh([t_0 f])$.
\end{lemma}
\begin{proof}
Given $\sigma \in \dCh([t_0(f)])$ and every $V \to \widebar Y$ we will use $\sigma^V_{V \times_{\widebar Y} \widebar X}$ to give a transformation $\widetilde \sigma^V_{V \times_Y X}$.  Using commutativity with proper pushforward, we can assume that $V$ is discrete.  One then obtains the diagram 

\begin{displaymath}
  \xymatrix{W \ar[rr]^{t_0(f)^\prime} \ar[dd]^{h^\prime} \ar[rd]^i & & V \ar[rd]^{id} \ar[dd]^(.7)h|\hole & \\ & W^\prime \ar[rr]^(.3){f^\prime} \ar[dd]^(.3){g^\prime}  & & V \ar[dd]^g \\ \widebar X \ar[rd]^{\iota_X} \ar[rr]^(.3){t_0(f)}|\hole &  & \widebar Y \ar[rd]^{\iota_Y} & \\ & X \ar[rr]^f & &  Y}
\end{displaymath}
All diagonal morphisms induce isomorphisms on $\dCh$.  Thus, define $\widetilde \sigma^V_{W^\prime}$ as $[i_\ast] \sigma^{V}_{W}$.  Likewise, for $\eta \in \dCh([f])$, we obtain $\widebar \eta^V_W \in \dCh([t_0(f)])$ via $[i_\ast]^{-1}\eta^
V_{W^\prime}$. It is straightforward to check that this operation defines a bivariant class (in both directions).  It clearly induces an isomorphism $\dCh([f]) \cong \dCh([t_0(f)]$. 
\end{proof}
This property, as examined above, is not clear for bivariant algebraic K-theory.  Following Lemma \ref{sameoverafield}, it is similar for structure morphisms, but remains true for all morphisms on operational theories since they are bootstrapped into a bivariant theory from the theory on the structure morphisms.  On the other hand, bivariant $K$-theory is not bootstrapped.   Clearly $\perf(f)$ and $\perf(t_0(f))$ are generally very far from equivalences, thus the question is how much of the ``algebraic'' difference in these categories is removed by taking their Grothendieck groups.

\subsubsection{The Todd class of a perfect complex}
\label{hot_toddy}
We briefly recall properties of the Todd class of a perfect complex.  Given a perfect complex on a discrete $V \in \dSch_{/B}$, we can assume that it is strict perfect, i.e., a bounded complex of vector bundles.  Define $td(\sP) = \prod td(\sP^i)^{(-1)^i} \in A^*(V)$.  The multiplicative nature of the Todd class of a vector bundles ensures this is well defined.  Given a perfect complex $\sP$ on a derived scheme $X$, we define $\Td(\sP) \in \dCh(id_X)$ as the unique operational class generated by the condition for any discrete $g: V \to X$, $\Td(\sP)^V_{V} = td(g^\ast \sP)$.  


\subsubsection{Verification of bivariant identities}

The verification of the bivariant identities are mostly trivial and (much like the discrete case) revolve around Lemma \ref{base_change}. 

\subsubsection{The orientation}

Let $\nu$ be the orientation on $\dCh$ defined by $\nu(f) = [f^!]_{vir}$ for $f: X \to Y$ quasi-smooth.  Similar to the case of $\theta$, this orientation is multiplicative.

\subsection{Operational derived Chow over $\mQ$}

Composing $\dCh$ with functor $Gr_{\mZ} \to Gr_{\mQ}$ obtained by tensoring with $\mQ$, we obtain a new bivariant theory.  We denote this theory as $\dCh_{\mQ}$.


\section{$dch$ and deformation to the normal cone}

\subsection{Derived local Chern character}
With $\dsch_{/B}$ as in the previous section, we now define the ``derived'' local Chern character. In the next section we will use it to build a morphism between bivariant K-theory and operational derived Chow with coefficients in $\mQ$.  

Let $f$ be a closed embedding. Then for any $\sF \in \perf(f)$, by Lemma \ref{globallyperfect} $f_\ast \sF \in \perf(Y)$ supported on $\widebar{X}$.  Given the homotopy fiber diagram  
\begin{displaymath}
  \xymatrix{W \ar[r] \ar[d] & V \ar[d]^g \\ X \ar[r]^f & Y}
\end{displaymath}
Define $dch^Y_X (\sF)^V_W: \dChQ(V) \to \dChQ(W)$ as $dch^Y_X(\sF)^{V}_{W} := \iota_{W \ast}(ch^{\widebar{V}}_{\widebar{W}} ((\iota_V \circ g)^\ast  f_\ast \sF) \cap \iota_{Y\ast}^{-1} - )$.  It is clear that $\iota_V^\ast g^\ast f_\ast \sF$ has as support $\widebar{W}$, thus the formula is well defined.

\begin{notation}
For clarity we sometimes include $f_\ast$ in the notation, i.e., $dch^Y_X(\sF)$ will be written $dch^Y_X(f_\ast \sF)$.
\end{notation}

\begin{remark}
This definition of the ``derived'' local Chern character is forced on any reasonable extension of the local Chern character: one can rephrase the derive local Chern character associated to a perfect complex as the unique operational class that acts as the underived local Chern character on discrete schemes.  
The main difference in its application to the Grothendieck-Riemann-Roch theory is that if $\sF \in \perf(f)$, then $f_\ast \sF$ is a perfect complex supported on $\widebar X$, but in most cases not the pushforward of a perfect complex on $\widebar X$. 
\end{remark}

\begin{proposition}
\label{invariant}
Let $f: X \to Y$ be a closed embedding and $\sF \in \perf(f)$.  Then $dch^Y_X(\sF) \in \dChQ(f)$.  Further, if $\sG \in \perf(f)$ is such that $[\sG] = [\sF]$, then $dch^Y_X(\sG) = dch^Y_X(\sF)$.  
\end{proposition}
\begin{proof}
To verify that it commutes with proper pushforward: let $h: V^\prime \to V$ be proper and $g: V \to Y$ and $\alpha \in \dChQ(V^\prime)$.  Let $\iota_{V^\prime\ast}\widebar \alpha = \alpha$.  Then we have the diagram with all squares Cartesian
\begin{displaymath}
  \xymatrix{ W^\prime \ar[r] \ar[d]^{h^\prime} & V^\prime \ar[d]^{h} \\ W \ar[r] \ar[d]^{g^{\prime}} & V \ar[d]^{g} \\ X \ar[r]^f & Y}
\end{displaymath}

\begin{align*}
  dch^Y_X(\sF) \cap h_\ast \alpha & = \iota_{W \ast}(ch^{\widebar{V}}_{\widebar{W}} ((\iota_V \circ g)^\ast  f_\ast \sF) \cap \widebar{h}_{\ast} \widebar \alpha) \\
& = \iota_{W\ast} (\widebar{h^\prime}_\ast(ch^{\widebar{V^\prime}}_{\widebar{W^\prime}} (t_0(h) \circ \iota_V \circ g)^\ast f_\ast \sF) \cap \widebar \alpha)\\
& =  \iota_{W \ast} (\widebar{h^\prime}_\ast(ch^{\widebar{V^\prime}}_{\widebar{W^\prime}} (\iota_{V^\prime}\circ h \circ g)^\ast f_\ast \sF) \cap \widebar \alpha)\\
& =  h^\prime_\ast (\iota_{W^\prime \ast}(ch^{\widebar{V^\prime}}_{\widebar{W^\prime}} (\iota_{V^\prime}\circ h \circ g)^\ast f_\ast \sF) \cap \widebar \alpha)\\
& =  h^\prime_\ast (dch^Y_X(\sF) \cap \alpha)
\end{align*}
The second line follows from \cite[Theorem 18.1]{fulton_intersection}, the others follow from the commutative squares the the adjunction morphisms $\iota_V$ provide.

For flat pullback, let $h: V^\prime \to V$ be flat. For $\widebar \alpha \in \dChQ(\widebar{V})$ with $\iota_{V \ast} \widebar \alpha = \alpha$, then $h^\ast(\alpha) = \iota_{V^\prime \ast} \widebar{h}^\ast \widebar \alpha$ (this follows from Lemma \ref{pushforward_formula}).  Thus,  
\begin{align*}
    dch^Y_X(\sF) \cap h^\ast \alpha & = \iota_{W^\prime \ast}(ch^{\widebar{V^\prime}}_{\widebar{W^\prime}} ((\iota_{V^\prime} \circ h \circ g)^\ast  f_\ast \sF) \cap \widebar{h}^\ast \widebar \alpha) \\
 &  = \iota_{W^\prime \ast}(\widebar{h^\prime}^\ast(ch^{\widebar{V}}_{\widebar{W}} ((\iota_{V} \circ g)^\ast  f_\ast \sF) \cap \widebar \alpha)) \\ & = h^\ast \iota_{W \ast} ch^{\widebar{V}}_{\widebar{W}} ((\iota_{V} \circ g)^\ast  f_\ast \sF) \cap \widebar \alpha) \\
& = h^\ast dch^Y_X \alpha
\end{align*}

We are left showing commutation with Gysin homomorphisms for quasi-smooth embeddings.  Applying Proposition \ref{gysin_commute}, we are reduced to the case that $Z = \mA^1$ and $Z^\prime = {0}$.  This follows from \cite{fulton_intersection} and our definition of $dch^{Y}_{X}(\sF)$ since $Z^\prime \to Z$ is now a regular embedding of (discrete) algebraic varieties.
The second statement of the theorem follows from the invariance of the local Chern character under equivalence.
\end{proof}

\subsection{Deformation to the normal cone}
\label{deformation_to_normal_cone}
\label{descent}

In this section, let $X \to Y$ be a quasi-smooth closed embedding of a derived scheme $X$ into a discrete $Y$.  Let $\dcone$ be the deformation to the normal cone and $C_{\widebar X} Y$ the normal cone of the embedding $\widebar X \to Y$. Both $\dcone$ and $C_{\widebar X} Y$ are discrete as derived schemes.  Let $p: C_{\widebar X} Y \to \widebar X$ and $\rho: \dcone \to Y \times \mP^1$ be the natural morphisms.  This gives rise to the  projection $\rho_1: \dcone \rightarrow Y$ and flat morphism $\rho_2: \dcone \to \mP^1$.  Recall that $C_{\widebar X} Y$ is a divisor of $\dcone$ with associated line bundle $\sO(1) := \rho_2^\ast \sO_{\mP^1}(1)$.  We denote $\rho_2^\ast \sO_{\mP^1}(n)$ as $\sO(n)$.  

The following sections detail the construction of a complex $\sO_{\widehat{X}}$ on $\dcone$  supported on $\widebar X \times \mP^1$ satisfying 
  \begin{enumerate}
  \item $\sO_{\widehat X} |_{\mP^1\backslash \{0\}} \cong \sO_{X \times \mP^1\backslash \{0\}}$
  \item $\sO_{\widetilde X} := \sO_{\widehat X} |_{0}$ is perfect and supported on $\widebar X$ (the zero section)
  \item $\sO_{\widetilde X} \cong \sO_{C_{\widebar X} Y} \otimes_{\sO_{\mV(\iota_X^\ast \mT_{X/Y})}} \sO_{\widebar {X}}$ where $\eta: C_{\widebar X} Y \to \mV(\iota_X^\ast \mT_{X/Y}[1])$ is the natural morphism induced by $X$.
  \end{enumerate}
The construction of $\sO_{\widehat X}$ proceeds via $\infty$-categorical $fppf$-descent.  We first show the local model and glue the local extensions to a global extension.  The notation $\sO_{\widehat X}$ is to suggest that one should think of this as the structure sheaf of a deformed derived scheme $\widehat X$.  However, we will not show nor need an algebra structure on $\sO_{\widehat X}$.

\subsubsection{The local case}
Recall the definition of $K(\sE; s)$ in \S \ref{derived_intro}.
\begin{proposition}
\label{extendit}
  Let $Y$ be an discrete affine scheme and $X \cong \spec K(\sE; s)$ for some locally free $\sE$ on $Y$.  Then there exists an affine scheme $\widehat X$ and morphism $\widehat X \to \dcone$ satisfying
  \begin{enumerate}
  \item $\widehat X |_{\mP^1\backslash \{0\}} \cong X \times \mP^1\backslash \{0\}$
  \item $\widehat X |_{0} \to C_{\widebar X} Y$ is a closed embedding supported on $\widebar X$ (the zero section)
  \item $\widehat X |_{0} \cong C_{\widebar X} Y \times^h_{\mV(\sE|_{\widebar X})} \widebar {X}$ where $\eta: C_{\widebar X} Y \to \mV(\sE)$ is the natural morphism induced by $s$.
  \item $\widehat X$ is well defined up to  equivalence of perfect complexes i.e. it does not depend on $\sE$.
  \end{enumerate}
\end{proposition}
\begin{proof}
We will refer to $K(\sE; s)$ as $\sO_{X_{\sE}}$.  With this terminology, $\sO_X \cong \sO_{X_{\sE}}$.  Let $\sF := \rho_1^\ast \sE$, then $\sF^\vee \cong \rho^\ast \sE^\vee$.  Let $\sK \subset \sO_{\dcone}$ be the sheaf of ideals defining $\widebar X \times \mP^1 \subset \dcone$. By definition of $s$, it is clear that for $\sF^\vee \xrightarrow{\rho_1^\ast s^{\vee}} \sO_{\dcone}$ has $\image \rho_1^\ast s^{\vee} \cong \sK \otimes \sO(-1) \subset \sO_{\dcone}$.   This morphism factors 
\begin{displaymath}
  \xymatrix{& \sF^\vee \ar[rd]^{\rho_1^\ast(s)^\vee} \ar[ld] & \\ \s{O}(-1) \ar[rr]^{\subset} & & \sO_{\dcone}} 
\end{displaymath}
and there exists a natural section of $s^\prime: \sO_{\dcone} \to \sF^\vee \otimes_{\sO_{\dcone}} \sO(1)$.  Define $\widehat X_{\sE} := \spec K(\sF \otimes_{\sO_{\dcone}} \sO(1); s^\prime)$.  It is clear that $\sO_{X_{\sE}} = K(\sF \otimes_{\sO_{\dcone}} \sO(1); s^\prime)$.

For the first property, note that $\mP^1 \backslash \{0\}$ is a trivializing Zariski open neighborhood for $\sO(1)$. Thus, $(\sF \otimes_{\sO_{\dcone}} \sJ)|_{\mP^1 \backslash \{ 0 \}} \cong \sF|_{\mP^1 \backslash \{ 0 \}}$ and  $s^\prime$ differs from $s$ by a unit.  This is enough to show  $\widehat{X}|_{\mP^1 \backslash \{0 \}} \cong \mA^1_{X}$.  
The second property is satisfied since $\image (s^\prime)^\vee = \sK$ by definition.  In other words, $\widehat X$ is supported on $\mP^1_{\widebar X}$.  This will continue to hold after restriction to the special fiber.

For the third property, let $i_0: C_{\widebar X} Y \to \dcone$ and $\sI \subset \sO_Y$ be the sheaf of ideals corresponding to $\widebar X$.  Then $C_{\widebar X} Y \times_{\dcone} \widehat X \cong \spec K(i_0^\ast(\sF \otimes_{\sO_{\dcone}} \sO(1)); i_0^\ast s^\prime)$; this follows from the fact $Y \times_X - $ preserves fiber products.  Let $\sV := \widebar{f}^\ast \sE$ ( i.e., $\sE|_{\widebar X}$) and $\pi: \mV(\sV) \to \widebar X$ be the natural projection.   We have the commutative diagram
\begin{displaymath}
  \xymatrix{ & \widebar X \ar[rd] \ar[ld] & \\ C_{\widebar X} Y \ar[rr]^{\eta} & & \mV(\sV)}
\end{displaymath}
where the lower morphism is induced by restricting the morphism $s^\vee:  \sE^\vee \to \sI \subset \sO_Y$ to $\widebar X$.  It is clear $C_{\widebar X} Y \times_{\mV(\sV)} \widebar X \cong \spec K(\eta^\ast \pi^\ast \sV; \eta^\ast r)$ where $r: \sO_{\mV(\sV)} \to \pi^\ast \sV$ is the identity section.  It is an easy calculation that $i_0^\ast (\sF \otimes_{\sO_{\dcone}} \sJ) \cong i_0^\ast \rho_1^\ast \sE \cong \eta^\ast \pi^\ast \sV$.  To show the isomorphism, it is enough to compare the two sections  $\eta^\ast r$ and $i^\ast_{0} s^\prime$.  Explicit computation shows that both $(\eta^\ast r)^\vee$ and $(i_0^\ast s^\prime)^\vee$ are the morphism which on the $n$th graded summand is the composite of the morphisms
\begin{displaymath}
\sE^\vee|_{\widebar X}\otimes_{\sO_{\widebar X}} \sI^n/\sI^{n+1} \xrightarrow{s\otimes id} \sI/\sI^2 \otimes \sI^n/\sI^{n+1} \rightarrow \s{I}^{n+1}/\s{I}^{n+2}.
\end{displaymath}
In other words, by construction, the section $s$ canonically yields the isomorphism.



We now show that $\sO_{\widehat X_{\sE}} \cong \sO_{\widehat X_{\sF}}$  for $\sO_{X_{\sE}} \cong \sO_{X_{\sF}}$, where $\sE$ and $\sF$ are locally free sheaves on $Y$.  Let $k_i: \sO(i) \to \sO(i +1)$ be the natural monomorphism and let $d_i^j$ denote the simplical edge morphism corresponding to $[j-1] \hookrightarrow [j]$,  missing $i$. With this notation in place, as simplicial $\sO_{\widebar X} Y$-modules, $(\sO_{\widehat X_{\sE}})_i = \rho^\ast (\sO_{X_{\sE}})_i \otimes_{\sO_{\dcone}} \sO(i)$ and by construction $(id_{X_{\sE}} \otimes k_i) \circ (d_i^j)_{\sO_{\widehat X_{\sE}}} = (d_i^j)_{\rho^\ast \sO_{X_{\sE}} \otimes \sO(i)}$.

Suppose $X \cong X_{\sG}$.  Since $X$ is quasi-isomorphic to both, $\sO_{X_{\sE}} \cong \sO_{X_{\sG}}$. Our assumption of $Y$ being affine means  $\sO_{X_{\sG}}$ and $\sO_{X_{\sE}}$ are fibrant-cofibrant (as simplicial $\sO_Y$-modules and algebras), thus this quasi-isomorphism is an morphism  $\phi: \sO_{X_{\sE}} \to \sO_{X_{\sG}}$ of simplicial algebras.  The cofibrant-fibrant nature of these modules implies $L\rho_1^\ast$ is underived when applied to them and thus, $\rho_1^\ast \phi: \rho_1^\ast \sO_{X_{\sE}} \to \rho_1^\ast\sO_{X_{\sG}}$.  We will show that $\phi$ can be naturally promoted to a morphism $\psi: \sO_{\widehat X_{\sE}} \to \sO_{\widehat X_{\sF}}$ (in any of the various model structures on $\QC(\dcone)$).  Using the above description of $(\sO_{\widehat X_{\sE}})_i$, let $\psi_i := (\rho^\ast \phi)_i \otimes id_{\sO(i)}$.   We must show $\psi$ commutes with the face and degeneracy maps.  

  Suppose $(d_i^j)_{\sO_{\widehat X_{\sG}}} \circ \psi_j - \psi_{j-1} \circ (d_i^j)_{\sO_{\widehat X_{\sE}}} \neq 0$ as a morphism $(\rho^\ast \sO_{X_{\sE}})_j \otimes \sO(j) \to (\rho^\ast \sO_{X_{\sG}})_{j -1}\otimes \sO(j-1)$.   Composing with $id \otimes k_{j -1}$, we have 
\begin{align*}
  id\otimes k_{j-1} \circ ((d_i^j)_{\sO_{\widehat X_{\sG}}} \circ \psi_j -
  \psi_{j-1} \circ (d_i^j)_{\sO_{\widehat X_{\sE}}}) & = id \otimes k_{j-1} \circ ((d_i^j)_{\sO_{\widehat X_{\sG}}} \circ \psi_j) - id \otimes k_j \circ \psi_{j-1} \circ (d_j^i)_{\sO_{\widehat X_{\sE}}}
\end{align*}
The equality  $(id \otimes k_{j-1}) \circ (\rho^\ast \phi_{j-1} \otimes id_{\sO(j-1)}) = \rho^\ast \phi_{j-1} \otimes id_{\sO(j)} \circ (id \otimes k_{j-1})$ and remarks above show this is equal to 
\begin{align*}
 (d_i^j)_{\rho^\ast \sO_{X_{\sE}} \otimes \sO(j)} \circ (\rho^\ast \phi_j \otimes id_{\sO(j)}) -  (\rho^\ast\phi_j \otimes id_{\sO(j)}) \circ (d_i^j)_{\rho^\ast \sO_{\widehat X_{\sE}}\otimes \sO(j)} 
\end{align*}
Yielding,
\begin{align*}
id\otimes k_{j-1} \circ ((d_i^j)_{\sO_{\widehat X_{\sG}}} \circ \psi_j -
  \psi_{j-1} \circ (d_j^i)_{\sO_{\widehat X_{\sE}}}) & = (\rho^\ast  ((d^j_i)_{\sO_{X_{\sG}}} \circ \phi_{j} - \phi_{i-1} \circ (d_i^j)_{\sO_{X_{\sE}}})) \otimes id_{\sO(j)} \\&= 0
\end{align*}
Since all terms are locally free and $k_{j-1}$ is a monomorphism, this is enough to give a contradiction.  Similar arguments give the case for degeneracy maps. 
It is clear that these arguments restrict to the natural affine open subsets $\dcone \backslash \rho_2^\ast(\infty)$ and $\dcone \backslash \rho_2^\ast(0)$.

It will be important in the next section to show that if $h: \sO_{X_{\sE}}\otimes \Delta_n \to \sO_{X_{\sG}}$ with $X_{\sE} \cong X_{\sG} \cong X$, then there exists a canonical lift $h^\prime: \sO_{\widehat X_{\sE}}\otimes \Delta_n \to \sO_{\widehat X_{\sF}}$.  The argument proceeds exactly as above.
  
\end{proof}

\subsubsection{Local to global}
For this section we assume $X \in \dSch_{/\spec k}$ where $k$ is a field.
Since $Y$ is quasi-projective, we can 1) assume $\sO_X$ is a bounded complex of vector bundles exact off of $\widebar X$,  2) one can find a basis $U_\alpha$ of affine Zariski open subschemes such that finite intersections are again affine.  Let $\{U_a\}_{a \in I}$ be an affine Zariski covering of $Y$ such $U_a = U_\alpha$ for some $\alpha$ and if $V_a = U_a \times^h_Y X$, then $V_a \cong \spec K(\sE_a; s_a)$; Proposition \ref{localnature} makes this possible.  By choosing $U_a$ appropriately, we can assume $\sO_{V_a}$ is a fibrant-cofibrant object of $\sO_{U_a}$-mod.  

For the following, let $U_0 := \coprod_{a \in I} U_a$ and $U_n := \overbrace{U_0 \times_Y \ldots \times_Y U_0}^{\textrm{$n+1$ times}}$ and $d^i_j: U_j  \to U_{j -1}$ be the $i$th face map (in the resulting simplicial object).   If we set $\widehat U_a := U_a \times_Y \dcone$ then $\widehat U_a = \widehat U_{a, 0} \cup \widehat U_{a, \infty} $ where $\widehat U_{a, \infty}$ is the complement of the principal divisor $\rho_{2}^\ast \{\infty\}$ (and similarly for $\widehat U_{a, 0}$). Thus, we can refine the Zariski cover $\coprod \widehat U_a$ of $\dcone$ to an Zariski affine cover $\coprod_{a \in I, j \in \{0, \infty\}} \widehat U_{a, j}$. 
The resulting simplicial object $\widehat U_\ast$ satisfies the same properties as $U_\ast$. 

It is well known that there exists a $fppf$-stack $\sM_{\perf(k)}$:
\begin{gather*}
  \sM_{\perf}: dSt_{k} \to \sS \\ \sM_{\perf}(X) = \sM_{\perf(X)}
\end{gather*}
where $\sM_{\perf(X)}$ is the $\infty$-groupoid (Kan complex) of objects of $\perf(X)$ and $\sS$ is the $\infty$-category of Kan complexes. For example, one can create such an object by Kan extending the functor
\begin{gather*}
  \sM_{\perf}: dAff_{k} \to \sS \\ \sM_{\perf}(\spec A) = N(\perf(\spec A)^{cf})
\end{gather*}
where the latter is the nerve of the full subcategory of the simplicial model category $\QC(\spec A)$ (with the projective model structure) consisting of fibrant-cofibrant perfect objects.   Applying $\sM_{\perf}$  to the simplicial affine scheme $U_{\ast}$ results in a cosimplicial $\infty$-groupoid.  Since $U_\ast$ is a $fppf$-hypercover of $Y$, this implies
\begin{displaymath}
  \sM_{\perf(Y)} \cong \holimproj\; N(\perf(U_*)^{cf})
\end{displaymath}
We will give explicit morphism of cosimplicial sets $\Delta \to \sM_{\perf(U_\ast)}$.  We then use this morphism and Proposition \ref{extendit} to construct a morphism $\Delta \to \sM_{\perf((\widehat U)_\ast)}$.  Descent then dictates that such a morphism corresponds to a perfect complex on $\dcone$.  To construct our morphism we use the adjunction pair $(N, \mathfrak{C})$ of \cite[Chapter 1]{topoi}.
  
Since $\sO_X$ is a bounded complex of vector bundles it canonically determines a $0$-simplex in each $\sM_{\perf(U_n)}$ (without strictness there is a choice of isomorphism involved). We denote the $0$-simplex in $\sM_{\perf(U_0)}$ determined by $\sO_X$ and $\coprod K(\sE_a; s_a)$ by $g_0$ and $h_0$, respectively.  The choice of isomorphism $K(\sE_a; s) \cong \sO_X|_{U_a}$ gives the following data
\begin{align*}
g_: \mathfrak{C} \Delta_1 &\to \perf(U_0)\\
g|_{\mathfrak{C}\Delta_1^\circ} & = (\coprod K(\sE_a; s_a), \sO_X|_{U_0}) 
\end{align*}
(This is equivalent to a point in $\Map_{\perf(U_0)}(\coprod K(\sE_a; s_a), \sO_X|_{U_0})$).  Since we are working with fibrant-cofibrant objects, $g$ has a homotopy inverse $\widetilde g$ with homotopy 
\begin{displaymath}
G: \Delta_1 \to \Map_{\perf(U_0)}(\sO_X|_{U_0}, \sO_X|_{U_0})
\end{displaymath}  
We now define $h_i: \Delta_i \to \sM_{\perf(U_i)}$ as $Ng_n$ where $g_n: \mathfrak{C}\Delta_n \to \perf(U_n)$.  Let
\begin{gather*}
  G_0: \Delta_0 \to \Map_{\perf(U_1)}(d_0^{1 \ast} g_0, d_1^{1 \ast} g_0)\\
 G_0 := d_0^{1 \ast}(\widetilde g) \circ d_0^{0 \ast}(g)\\
  g_1(j) = d_j^{1 \ast} g_0\\
  g_1(\Map_{\mathfrak{C}\Delta_1}(0, 1)) = G_0
\end{gather*}
To define $g_n$, $n > 1$, we first introduce additional notation.  The zero simplices of $\Delta_n$ correspond to morphisms $[0]\to [n]$.  To $j \in [n]$ there is a unique composite $\partial_{j,n}^0: d^n_{j_n}d^{n-1}_{j_{n-1}} \ldots d^0_{j_0}$ satisfying $n \geq j_n \geq j_{n-1} \geq \ldots \geq j_0 \geq 0$ (this corresponds to $[0] \to [n]$ with image $j$).  More generally, let $\partial^i_{j, n}$ be the simplicial morphism correspond to the linear embedding of $[i] \to [n]$ that sends $0 \to j, 1 \to j+1$, etc.  Explicitly, $\partial^i_{j,n}: d_n^n\circ d^{n-1}_{n-1} d_{i+j+1}^{i+j+1}d_{0}^{i+j}\ldots d^{i+1}_0 d^i_0$

Set
\begin{gather*}
  G_1: \Delta_1 \to \Map(\partial^0_{0, 2}, \partial^0_{2,2})\\
  G_1 :=   \partial^0_{2,2}\widetilde g \circ G \circ \partial^0_{0,2} g
\end{gather*}
It is clear that $G_1$ is a homotopy from $(d^1_2)^\ast h_1 \circ d(^1_0)^\ast h_1$ to $(d^1_1)^\ast h_1$.  In other words, in the homotopy category, it ensures the standard cocycle condition is satisfied. Let
\begin{gather*}
G_i:  \times^{i-1} \Delta_1 \to \Map_{\perf(U_n)}((\partial^0_{0, i})^\ast(\coprod K(\sE_a; sa)), (\partial^0_{i, i})^\ast \coprod K(\sE_a; sa)) \\
  G_i := (\partial^1_{i-1, i})^\ast G_0 \circ \ldots \circ (\partial^1_{2,i})^\ast G_0 \circ (\partial^2_{0, i})^\ast G_1 \times (\partial^1_{i-1, i})^\ast G_1 \circ \ldots \\ \ldots \circ (\partial^1_{3,i})^\ast G_0 \circ (\partial^2_{1, i})^\ast G_1 \circ (\partial^1_{0, i})^\ast G_0 \times \ldots  \times (\partial^2_{i-2, i})^\ast G_1 \circ  (\partial^1_{i-3, i})^\ast G_0 \circ \ldots \circ (\partial^1_{0,i})^\ast G_0
\end{gather*}

The objects of $\mathfrak{C}\Delta_n$ correspond $0$-simplices of $\Delta_n$ and $\Map_{\mathfrak{C}\Delta_n}(j, j + k) = \times^{k-1} \Delta_1$.  Define $g_n$ as 
\begin{gather*}
  g_n(j) = (\partial^0_{j, n})^\ast(\coprod K(\sE_a; s_a))\\
  g_n(\Map_{\mathfrak{C}\Delta_n}(j, j + k)) = (\partial^k_{j, n})^\ast G_k 
\end{gather*}
It is important to note that we have only repackaged the data naturally given by $\sO_X$ into a more desirable local description. 

 We now give a morphism $\widehat h: \Delta \to \sM_{\perf(\widehat U_*)}$.  Using Proposition \ref{extendit}, there exists  
\begin{gather*}
  \widehat G_i: \times^{i-1} \Delta_1 \to \Map_{\perf(\widehat U_n)}((\partial^0_{0, i})^\ast(\coprod \sO_{\widehat X_{\sE_a}}), (\partial^0_{i, i})^\ast(\coprod \sO_{\widehat X_{\sE_a}}))
\end{gather*}
agreeing with $\rho_1^\ast G_i$ outside of the special fiber.  
fine $\widehat g_n: \mathfrak{C}\Delta_n \to \perf(\widehat U_n)$ as 
\begin{gather*}
  g_n(j) = (\partial^0_{j, n})^\ast(\coprod \sO_{\widehat X_{\sE_a}})\\
  g_n(\Map_{\mathfrak{C}\Delta_n}(j, j + k)) = (\partial^k_{j, n})^\ast \widehat G_k
\end{gather*}
Last, let $(\widehat h)_n = N \widehat g_n$.  From the construction in Proposition \ref{extendit}, it is clear that this will retain the same identifies as  $h$ (i.e. $(\widehat h)_i$ fills in the ``hole'' induced by restricting $(\widehat h)_{i-1}$).  Thus, by descent, there exists an $\sO_{\widehat X} \in \perf(\dcone)$ that clearly satisfies conditions (1) and (2) from \S~\ref{deformation_to_normal_cone}. 

To prove that it satisfies the third condition, first note that the morphism $\iota_X^\ast \mL_{X/Y} \to \sI/\sI^2$ is induced by taking $\pi_1$ in the following distinguished triangle 
\begin{displaymath}
\xymatrix{  \iota_X^\ast \mL_{X/Y} \ar[rr] & & \mL_{\widebar X/Y} \ar[ld] \\ & \ar@{-->}[lu]^{[1]} \mL_{\widebar X/X}}
\end{displaymath}
This is locally given by a choice of vector bundle $\sE$ on $Y$ and a morphism $\sE^\vee \to \sI$, (i.e., a section of $\sE$).  Although different choices of vector bundles on $Y$ may not be isomorphic, their restriction to $\widebar X$ are isomorphic.  This in turn induces the embedding $C_{\widebar X} Y \hookrightarrow \mV(\iota_X^\ast \mT_{X/Y}[1])$.
Since $\widebar X \to \mV(\iota_X^\ast \mT_{X/Y}[1])$ is a regular embedding, quasi-projectivity ensures we can globally resolve $\sO_{\widebar X}$ using the Bar resolution.  This then gives an explicit representative for $\sO_{C_{\widebar X} Y} \otimes_{\sO_{\mV(\iota_X^\ast \mT_{X/Y}[1])}} \sO_{\widebar X}$.  The result is now clear since locally $\sO_{\widehat X}$ and $\sO_{C_{\widebar X} Y} \otimes_{\sO_{\mV(\iota_X^\ast \mT_{X/Y}[1])}} \sO_{\widebar X}$ are equal and their patching data are identical, i.e., the data is identical when restricting to the cosimplicial space $\sM_{\perf(V_n)}$ (the cosimplicial $\infty$-groupoid of objects), where $V_n = U_n \times_Y \widebar X$.  Descent then dictates that $\sO_{\widehat X|_{0}} \cong \sO_{C_{\widebar X} Y} \otimes_{\iota_X^\ast \mL_{X/Y}} \sO_{\widebar X}$.



\begin{remark}
  Although we have only shown that there exists an extension as a perfect complex, it is possible to show it retains an algebra structure as well.  This allows the rephrasing of the above conditions as an
  embedding $\widehat X \to \dcone$ satisfying
  \begin{enumerate}
    \label{properties}
  \item $\widehat X |_{\mA^1\backslash \{0\}} \cong X \times
    \mA^1\backslash \{0\}$
  \item $\widetilde X := \widehat X |_{0} \to C_{\widebar X} Y$ is a
    closed embedding supported on $\widebar X$ (the zero section)
  \item $\widetilde X \cong C_{\widebar X} Y \times^h_{\mV(\iota_X
      \mL_{X/Y}[1])} \widebar {X}$ where $C_{\widebar X} Y \to
    \mV(\iota_X \mL_{X/Y})$ is the natural morphism induced by $X$.
  \item $\widehat X \to \dcone$ is quasi-smooth embedding
    and $\sO_{\widehat X}$ is perfect over $\dcone$.
  \end{enumerate}
  We have not added this layer due to increased complexity in the descent argument.  We will only need the perfect complex.
\end{remark}

\begin{remark}
  Although we only provide a deformed perfect complex for $\sO_X$, it is our belief that there exists a generalized construction that allows one to deform any $\sF \in \perf(f)$ for any closed embedding $f$.  This type of construction would be useful for understanding ``higher'' obstruction theory.
\end{remark}

\subsection{$dch^Y_X(\sO_X)$ for  quasi-smooth embeddings.}

In the case that $f$ is a quasi-smooth embedding, the following equality of operational classes is integral to our comparison of orientations and invariance of $\tau$ done in later sections.
\begin{proposition}
\label{dch_explicit}
  Let $f: X \to Y$ be a  quasi-smooth embedding.  Then $dch_X^Y(\sO_X) = \Td(\mT_{X/Y}) \cap  [f^!]_{vir}$.
\end{proposition}
\begin{proof} Commutativity by both sides with proper pushforward shows it is enough to test the equality on the class $[V]$ with $V \to Y$ discrete. We first reduce to the case that $Y$ is discrete.  Let $W$ fit into the homotopy fiber diagram
\begin{displaymath}
  \xymatrix{ W \ar[r] \ar[d] & V \ar[d]^{j} \\ X \ar[r]^f & Y}
\end{displaymath}
Then $j$ factors as $\iota_Y \circ \widebar j$ with $\widebar j: V \to \widebar Y$. This yields the homotopy Cartesian diagram 
 \begin{displaymath}
   \xymatrix{W \ar[r] \ar[d]^{k} &  V  \ar[d]^{\widebar j} \\ Z \ar[r]^g \ar[d]^h & \widebar Y \ar[d]^{\iota_Y} \\ X \ar[r]^f & Y}
 \end{displaymath}
Since the lower horizontal morphism is affine, it is perfect and we can apply base change.
\begin{align*}
  dch^Y_X(\sO_X) \cap [V] &= \iota_{W \ast}(ch^V_{\widebar W}( j^\ast f_\ast \sO_X) \cap [V]) \\
  & = \iota_{W \ast}(ch^V_{\widebar W}( \widebar{j}^\ast \iota_Y^\ast f_\ast \sO_X) \cap [V]) \\
& = \iota_{W \ast}(ch^V_{\widebar W}( \widebar{j}^\ast g_\ast h^\ast \sO_X) \cap [V])\\
& = \iota_{W \ast}(ch^V_{\widebar W}( \widebar{j}^\ast g_\ast \sO_Z) \cap [V])\\
& = dch^{\widebar Y}_{Z}(\sO_Z) \cap [V]
\end{align*}
For the other side of the equation, since $h^\ast \mL_{X/Y} \cong \mL_{Z/\widebar Y}$, and from \S~\ref{hot_toddy}, 
$\mr{Td}(\mL_{X/Y}) \cap [V] = \mr{Td}(\mL_{Z/\widebar Y}) \cap [V]$.  Further, since $[f^!]_{vir}$ is operational, we have gives  $[f^!]_{vir} [V] = [g^!]_{vir} [V]$.   Thus, we can replace $X$ with $Z$ and assume that $Y$ is discrete.

We can now assume $j: V \to Y$ with both $V$ and $Y$ discrete,  $f: X \to Y$ quasi-smooth, and $W$ fit into the homotopy Cartesian diagram 
\begin{displaymath}
  \xymatrix{ W \ar[r] \ar[d] & V \ar[d]^j \\  X \ar[r]^f & Y}
\end{displaymath}
Let $p: M^\circ_{\widebar W} V \to \dcone$, $i_1: Y \to \dcone$ and $i_0: C_{\widebar X} Y \to \dcone$.  
Since $i_0$ and $i_1$ are quasi-smooth closed embeddings they commute with $dch^{\dcone}_{\mP^1_{\widebar X}}(\sO_{\widehat X}) \cap -$.  Thus, 
\begin{align*}
i_1^!  (dch^{\dcone}_{\mP^1_{\widebar X}}(\sO_{\widehat X}) \cap [M^\circ_{\widebar W} V]) &= dch^Y_{X}(i_1^\ast \sO_{\widehat X}) \cap [i_1^!] \cdot [M^\circ_{\widebar W} V] \\  
& = dch^Y_{X}(\sO_{X})\cap [V].
\end{align*}
On the other hand, using notation from \S~\ref{deformation_to_normal_cone}, 
\begin{align*}
i_0^!  (dch^{\dcone}_{\mP^1_{\widebar X}}(\sO_{\widehat X}) \cap [M^\circ_{\widebar W} V]) &= dch^{C_{\widebar X}Y}_{\widebar X}(i_0^\ast \sO_{\widehat X}) \cap [i_0^!] \cdot [M^\circ_{\widebar W} V] \\  
& = dch^{C_{\widebar X} Y}_{\widebar X}(\sO_{\widetilde X}) \cap [C_{\widebar W} V].
\end{align*}
$\mA^1$-invariance of the local Chern character ensures an equality, thus
\begin{align*}
 dch^Y_{X}(\sO_{X})\cap [V] &=   i_1^!  (dch^{\dcone}_{\mP^1_{\widebar X}}(\sO_{\widehat X}) \cap [M^\circ_{\widebar W} V])  \\
& = i_0^!  (dch^{\dcone}_{\mP^1_{\widebar X}}(\sO_{\widehat X}) \cap [M^\circ_{\widebar W} V]) \\
& = dch^{C_{\widebar X} Y}_{\widebar X}(\sO_{\widetilde X}) \cap [C_{\widebar W} V].
\end{align*}
We have the following homotopy Cartesian diagram
 \begin{displaymath}
   \xymatrix{ \widetilde X \ar[r]^{j^\prime} \ar[d]^{\mu} & C_{\widebar X} Y \ar[d]^{\nu} \\ \widebar X \ar[r]^{k\qquad} & \mV(\iota_X^\ast \mT_{X/Y}[1])}
 \end{displaymath}
We are being slightly abusive, the equivalence of property (3) of $\sO_{\widetilde X}$ ensures it is the structure sheaf of a derived scheme we are denoting by $\widetilde X$. Note that $dch^{C_{\widebar X} Y}_{\widetilde X}(\sO_{\widetilde X})\cap [C_{\widebar W} V] = \iota_{\widetilde X \ast} dch^{C_{\widebar X} Y}_{\widebar X}(\sO_{\widetilde X}) \cap [C_{\widebar W} V]$.  Base change ensures $j^\prime_{\ast}\sO_{\widetilde X} \cong \nu^\ast k_\ast \sO_{\widebar X}$. 

Since $\nu$ is proper, commutativity with proper pushforward shows
\begin{align*}
dch^{C_{\widebar X} Y}_{\widebar X}(\sO_{\widetilde X}) \cap [C_{\widebar W} V]  & = \mu_{\ast}  (dch^{C_{\widebar X} Y}_{\widetilde X}(\sO_{\widetilde X})\cap [C_{\widebar W} V]) \\
&= dch^{\mV(\iota_X^\ast \mT_{X/Y}[1])}_{\widebar X}(\sO_{\widebar X}) \cap \nu_\ast [C_{\widebar W} V]
\end{align*}
However, the right hand is in discrete algebraic geometry and thus by \cite[Corollary 18.1.2]{fulton_intersection}
\begin{align*}
dch^{\mV(\iota_X^\ast \mT_{X/Y}[1])}_{\widebar X}(\sO_{\widebar X}) \cap \nu_\ast [C_{\widebar W} V] & = ch^{\mV(\iota_X^\ast \mT_{X/Y}[1])}_{\widebar X}(\sO_{\widebar X}) \cap \nu_\ast [C_{\widebar W} V]\\
& = \mr{Td}(\iota_X^\ast \mT_{X/Y}[1])^{-1} \cdot [k^!]_{vir} \cap \nu_\ast [C_{\widebar W} V]\\
&=  \mr{Td}(\iota_X^\ast \mT_{X/Y}[1])^{-1} \cdot \iota_{X\ast}^{-1}  ([f^!]_{vir} \cap [V])
\end{align*}
The last equality is the definition of $[f^!]_{vir}$.  
Using that
\begin{displaymath}
\iota^{-1}_\ast (\mr{Td}(\mT_{X/Y}[1])^{-1} \cap \alpha) = \mr{Td}(\iota_X^\ast \mT_{X/Y}[1])^{-1} \cap \iota_{X\ast}^{-1} \cdot  \alpha
\end{displaymath}
for any $\alpha \in \dCh(X)$ gives the equality
 \begin{align*}
dch^Y_{\widebar X}(\sO_X) \cap [V] & =dch^{C_{\widebar X} Y}_{\widebar X}(\sO_{\widetilde X})\cap [C_{\widebar W} V] \\ &=  \iota_{X\ast}^{-1}( \mr{Td}(\mT_{X/Y}[1])^{-1} \cdot [f^!]_{vir} \cap [V])
 \end{align*}
Applying $\iota_{X\ast}$ to this yields the result.

\end{proof}

\section{$\tau$: the morphism between the two theories}

From Proposition \ref{invariant}, for $f$ a closed embedding, there exists a well defined homomorphism $\tau_f: K(f) \to \dChQ(f)$ defined by $\tau_f([\sF]) = dch^Y_X(\sF)$.  To extend $\tau$ to any morphism we apply Lemma \ref{factorizeit}.  Choosing a factorization  
\begin{displaymath}
  \xymatrix{X \ar[rr]^i \ar[rd]^f & & S \ar[ld]^s \\ & Y & }
\end{displaymath}
with $s$ smooth and $i$ a closed embedding, by Lemma \ref{smoothperfect} $i_\ast \sF$ is perfect. We then 
define $\tau_f([\sF]) := \mr{Td}(i^\ast \mT_{S/Y}) \cap dch^{S}_{X}(\sF) \cap [s^*]$.  In the bivariant language,  $\tau_f([\sF])$ is the image of $(\mr{Td}(i^\ast \mT_{S/Y}), dch^S_X(\sF),  [s^\ast])$ under the natural bivariant multiplication:
\begin{displaymath}
   \dChQ(X \xrightarrow{id} X) \otimes \dChQ(X \xrightarrow{f} S) \otimes \dChQ(S \xrightarrow{s} Y) \to \dChQ(X \to Y).
\end{displaymath}
\begin{remark}
 Since Chern classes commute with bivariant operations (for the same reasons as \cite[Proposition 17.3.2]{fulton_intersection}), we could equivalently define $\tau_f([\sF])= dch^S_X(\sF) \cap \Td(\mT_{S/Y}) \cap [s^\ast]$.  Both are useful in certain contexts. 
\end{remark}
\begin{lemma}
  $\tau_f([\sF])$ is independent of the factorization used to define it.
\end{lemma}
\begin{proof}
For this proof, since $[f]$ is not changing,  we will denote $\tau_S$ as the operational element corresponding to a factorization $X \to S \to Y$.  The setup yields the commutative diagram  
\begin{displaymath}
  \xymatrix{X \ar@{-->}[rd]^k  \ar[rdd] \ar[rrd] & & \\ & S \times_Y S^\prime \ar[r]_h \ar[d]^{h^\prime} & S \ar[d]^{s} \\ & S^\prime \ar[r]^{s^\prime} & Y}
\end{displaymath}
with all vertical and horizontal morphisms smooth, and all diagonal morphisms closed embeddings.   By symmetry we can replace our factorization  $X \to S^\prime \to Y$ with $X \to S \times_Y S^\prime \to Y$.  Similar to \cite[Proposition 18.3.1]{fulton_intersection} one then has a fiber diagram

\begin{displaymath}
  \xymatrix{ X \ar[r]& X\times_Y S^\prime \ar[d] \ar[r]& S \times_Y S^\prime \ar[d] \\ & X \ar[r] & S}
\end{displaymath}
Since $\tau_S$ and $\tau_{S \times_Y S^\prime}$  are composed of operational elements, they both commute with flat pullback and proper pushforward.  Thus, it is enough to test $\tau_S$ and $\tau_{S \times_Y S^\prime}$ against $[V] \in \dCh(V)$ for $g: V \to Y$ discrete. 

Pulling back the above diagram along the morphism $g$ yields a fiber diagram  
\begin{displaymath}
  \xymatrix{ W \ar[r]^(.4)h & W \times_V Z^\prime \ar[d]^q \ar[r]^j& Z \times_V Z^\prime \ar[d]^p \\ & W \ar[r]^i & Z}
\end{displaymath}
with $W := X \times_Y V$, $Z := S \times_Y V$, and $Z^\prime := S^\prime \times_Y V$.  The smoothness of $S \to Y$ implies $Z \cong \widebar{Z}$ and $Z^\prime \cong \widebar{Z^\prime}$.  Note that $h$ is a quasi-smooth embedding.  Since $dch$ is operational, commutativity with flat pullback means
\begin{align*}
  [q^\ast]\cdot dch^Z_W(i_\ast \sF)\cap [Z] = dch^{Z \times_V Z^\prime}_{W \times_V Z^\prime}(p^\ast i_\ast \sF) \cap [Z \times_V Z^\prime]
  & = dch^{Z \times_V Z^\prime}_{W \times_V Z^\prime}(j_\ast q^\ast \sF) \cap [Z \times_V Z^\prime]
\end{align*}

Now, although commutativity of $\tau$ with products will be shown Theorem \ref{dch_bivariant}, the arguments used carry over directly to this case (using that closed embeddings are projective) and Proposition \ref{dch_explicit} implies
\begin{align*}
dch^{Z \times_V Z^\prime}_{W}(h^\ast q^\ast \sF\otimes_{\sO_{W}} \sO_W) & = dch^{W \times_V Z^\prime}_{W}(\sO_W) \cap dch^{Z \times_V Z^\prime}_{W \times_V Z^\prime}(q^\ast \sF)\\ & = \Td(\mT_{W/W \times_V Z^\prime}) \cap [ h^!]_{vir} \cdot [q^\ast] \cdot dch^Z_W(\sF)\cap [Z]
\end{align*}
Since $q$ is smooth, $[q^\ast] = [q^!]_{vir}$ with trivial obstruction theory. 
Using functorality of virtual Gysin homomorphisms and noting that $h^\ast q^\ast \sF \cong \sF$, 
\begin{align*}
  dch^{Z \times_V Z^\prime}_{W}(\sF) \cap [Z \times_V Z^\prime] &= \Td(\mT_{W/W \times_V Z^\prime}) \cap [ (q \circ h)^!]_{vir} \cdot dch^Z_W(\sF)\cap [Z]\\
 &= \Td(\mT_{W/W \times_V Z^\prime}) \cap dch^Z_W(\sF)\cap [Z]
\end{align*}
Multiplying both sides by $\Td((j \circ h)^\ast \mT_{Z \times_V Z^\prime})$ the result will follow once we show
\begin{align*}
  \Td((j \circ h)^\ast \mT_{Z \times_V Z^\prime}) \cap \Td(\mT_{W/W \times_V Z^\prime}) \cong \Td(i^\ast \mT_{Z/V})
\end{align*}
This follows by multiplicativity of the Todd class, dualizing the distinguished triangles
\begin{displaymath}
  \xymatrix{ h^\ast \mL_{W \times_V Z^\prime/W} \ar[rr] & & 0 \ar[ld] \\ & \mL_{W/W \times_V Z^\prime} \ar@{-->}[lu]^{[1]} &}
\end{displaymath}
\begin{displaymath}
  \xymatrix{ p^\ast \mL_{Z/V} \ar[rr] & & \mL_{Z \times_V Z^\prime/V} \ar[ld] \\ & \mL_{Z \times_V Z^\prime/Z}  \ar@{-->}[lu]^{[1]} &}
\end{displaymath}
and the isomorphism $\mL_{W \times_V Z^\prime/W} \cong j^\ast \mL_{Z \times_V Z^\prime/Z}$.

\end{proof}


Now that we know it is well defined, we must verify that it is a morphism of bivariant theories.

\begin{theorem}
\label{dch_bivariant}
  $\tau: K \to \dChQ$ is a morphism of bivariant theories.
\end{theorem}
\begin{proof}

We must check that $\tau$ commutes with the product, proper pushforward, and pullback.  The proof generally follows the analogous proof, as carried out in \cite{fulton_intersection}.

We begin with pullback.  Given $X \xrightarrow{f} Y$, $Y^\prime \xrightarrow{g} Y$, factorization of $f$ induces the diagram
\begin{displaymath}
  \xymatrix{ X^\prime \ar[r]^{i^\prime} \ar[d]^{g^{\prime \prime}} & S^\prime \ar[r]^{s^\prime} \ar[d]^{g^\prime} & Y^\prime \ar[d]^g \\ X \ar[r]^i  & S \ar[r]^s &  Y}
\end{displaymath}
with all squares Cartesian.  As usual,  $s \circ i = f$, $i$ a closed embedding and $s$ is smooth.  Stability under base change implies the same for $i^\prime$ and $s^\prime$.  Thus, we have a factorization of $f^\prime = s^\prime \circ i^\prime$.  Given $\sF \in \perf(f)$, by definition of $\tau_f([\sF])$ and pullback of operational classes
\begin{align*}
  g^\ast \tau_f([\sF]) &= g^{\second \ast} \Td(i^\ast \mT_{S/Y}) \cap  g^{\prime \ast} dch^{S}_{X}((g^\prime )^\ast i_\ast \sF) \cap  g^\ast [s^{\ast}]\\
&= \Td(g^{\second \ast} i^\ast \mT_{S/Y}) \cap dch^{S^\prime}_{X^\prime}( g^{\prime \ast} i_\ast \sF)\cap  [s^{\prime \ast}]\\
&= \Td(i^{\prime\ast} g^{\prime ^\ast} \mT_{S/Y}) \cap dch^{S^\prime}_{X^\prime}( i^\prime_\ast g^{\second \ast} \sF)\cap  [s^{\prime \ast}]\\
& = \Td(i^{\prime\ast} \mT_{S^\prime/Y^\prime}) \cap dch^{S^\prime}_{X^\prime}(g^{\second \ast} \sF) \cap  [s^{\prime \ast}]\\
& = \tau_{f^\prime}([g^{\prime \ast} \sF])\\
& = \tau_{f^\prime}(g^\ast [\sF])
\end{align*}


Throughout the rest of this proof, let $Y \xrightarrow{q} B$ be the structure morphism, then $\mP_Y(\sE) := Y \times_{B} \mP(\sE)$ for some locally free sheaf $\sE$ on $B$.  We also will denote $\sO_Y(n) := q^{\prime \ast} \sO_{\mP(\sE)}(n)$, where $q^\prime: \mP_Y(\sE) \to \mP(\sE)$ is the natural morphism induced from $q$.  It will be relevant to the proof that if $\pi^\prime: \mP_Y(\sE) \to Y$ is the natural projection, then $\pi^\prime_\ast \sO_Y(n) \cong Sym^n(q^\ast \sE)$; this follows from proper base change.

Let $f, g$ be a composable pair of morphisms and assume $f$ is proper.  The proof of Lemma \ref{factorizeit}  shows $f$ factors 
 \begin{align*}
   \xymatrix{X \ar@{^{(}->}[r]^i \ar[rd] & U \times_S Y \ar[d] \ar[r]^{j} & \mP_Y(\sE) \ar[ld]^{\pi^\prime} \\ & Y&}
 \end{align*}
Since $\pi^\prime$ is separated, $j \circ i$ is proper \cite[Exercises 4.8]{Hartshorne}, and thus a closed immersion. By Lemma \ref{closedequiv}, if $\sF \in \perf(f)$, $(i \circ j)_\ast \sF$ is perfect and we can replace our factorization $X \to U \to Y$ with $X \to \mP_Y(\sE) \to Y$.  Choosing a global factorization of $g$ we obtain the commutative diagram, with the square Cartesian
\begin{displaymath}
  \xymatrix{ X \ar[r]^i \ar[rd]^f & \mP_Y(\sE) \ar[r]^j \ar[d]^{\pi^\prime} & \mP_S(\sE) \ar[d]^{\pi} \\ & Y \ar[r]^k \ar[rd]^g & S \ar[d]^\eta \\ & & Z}
\end{displaymath}
All horizontal morphisms are closed immersions, the vertical morphisms are smooth,  and $\pi$ proper.  This diagram allows us to split proper morphisms into two simpler cases: closed embeddings and projective fibrations.  For composable embeddings it is a consequence of the following property of the local Chern character:
\begin{displaymath}
  ch^{\widebar Z}_{\widebar Y}(\sF) = \widebar{f}_{\ast} ch^{\widebar Z}_{\widebar X}(\sF) 
\end{displaymath}
for a perfect complex on $\widebar Z$ exact off of $\widebar Y$ (and thus, exact off of $\widebar X$).

We can now assume that $i = id$ in the above diagram.  We first need a lemma.
\begin{lemma}
  $K(g \circ \pi^\prime)$ is generated by classes $\sO_Y(n) \otimes_{\sO_{\mP_{Y}(\sE)}} \pi^{\prime \ast} \sG$, where $\sG \in \perf(g)$.
\end{lemma}
\begin{proof}
Let $\Delta: \mP(\sE) \to \mP(\sE) \times \mP(\sE)$ be the diagonal morphism and let $\mP(\sE)_\Delta$ denote $\mP(\sE)$ as a closed subscheme of $\mP(\sE) \times \mP(\sE)$.  Since $B$ is discrete, the equivalence 
\begin{displaymath}
\mP(\sE)_{\Delta} \cong  (\mP(\sE) \times \mP(\sE)) \times_{\mV(\sO(1)\boxtimes \mT_{\mP(\sE)/B}(-1))} (\mP(\sE) \times \mP(\sE))
\end{displaymath}
 where the fiber product is over the the $0$-section and the dual of the natural morphism $\mV(\sO(-1)\boxtimes \Omega(1)) \to \sO_{\mV(\sO(-1)\boxtimes \Omega(1)}$. Since $\mP_Y(\sE) \to \mP_Y(\sE) \times_Y \mP_Y(\sE)$ is the pullback of the $\Delta$, using similar notation as above, $\sO_{\mP_Y(\sE)_\Delta} \cong K(\sO_Y(1)\boxtimes \mathbb{T}_{\mP_Y(\sE)/Y}(-1); s)$.  With this notation in place, the argument proceeds much like \cite[\S 8.3]{Huybrechts} (making the appropriate conversions between simplicial and dg-notation, e.g. bad truncation corresponds to taking skeleta).  In particular, one obtains a series of fiber-cofiber diagrams
\begin{gather*}
  \xymatrix{\sO(-n)\boxtimes\bigwedge^n\mL_{\mP_{Y}(\sE)/Y}(n) \ar[r] \ar[d] & \sO(-(n-1))\boxtimes\bigwedge^{n-1}\mL_{\mP_{Y}(\sE)/Y}(n) \ar[d] \\ 0 \ar[r] & M_{n}}\\
  \xymatrix{M_{i+1} \ar[r] \ar[d] & \sO(-i)\boxtimes\bigwedge^{i}\mL_{\mP_{Y}(\sE)/Y}(i) \ar[d] \\ 0 \ar[r] & M_{n}}\\ 
  \xymatrix{M_{1} \ar[r] \ar[d] & \sO_{\mP_Y(\sE) \times \mP_{Y}(\sE)} \ar[d] \\ 0 \ar[r] & \sO_{\mP_{Y}(\sE)_{\Delta}}}
\end{gather*}

It is well known the Fourier-Mukai functor (integral transform) $\Phi_{\sO(-i)\boxtimes\bigwedge^{i}\mL_{\mP_{Y}(\sE)/Y}(i)} \in \mr{End}(\perf(g \circ \pi^\prime))$ acts via the following 
\begin{displaymath}
\Phi_{\sO(-i)\boxtimes\bigwedge^{i}\mL_{\mP_{Y}(\sE)/Y}(i)} \cong (\pi^{\prime \ast}\pi^\prime_\ast (- \otimes \bigwedge^i\mL_{\mP_Y(\sE)/Y}(i)))\otimes \sO(-i)
\end{displaymath}
The fact that this preserves $\perf(g \circ \pi^\prime)$ is clear from this description.  For example, $\sO(n) \boxtimes \sG \in \perf(g \circ \pi)$ since it is a product by the same means as shown for products in bivariant K-theory..  Using induction, one has that $\perf(g \circ \pi^\prime)$ is generated (as a stable category) by elements of the form $\sO(n) \otimes_{\sO_X} \pi^{\prime \ast} \sG$ and this clearly stays true upon reduction to $K(g \circ \pi^\prime)$
\end{proof}

With this lemma, we are left proving the statement for $\sF \cong \sO(n) \otimes \pi^{\prime}(\sG)$.  
\begin{align*}
  \pi^\prime_\ast\tau_{g \circ \pi^\prime} ([\sF])) &= \pi^\prime_\ast(\Td(j^\ast \mT_{\mP_S(\sE)/Z}) \cap dch^{\mP_S(\sE)}_{\mP_Y(\sE)}(j_\ast (\sO_Y(n) \otimes \pi^{\prime\ast}(\sG)))  \cap [\eta \circ \pi^*]) \\
 & =  \pi^\prime_\ast(\Td(j^\ast \mT_{\mP_S(\sE)/Z}) \cap dch^{\mP_S(\sE)}_{\mP_Y(\sE)}((\sO_S(n) \otimes j_\ast \pi^{\prime\ast}(\sG)))  \cap [\eta \circ \pi^\ast]) 
\\ 
&= \pi^\prime_\ast( \Td(j^\ast \mT_{\mP_S(\sE)/Z}) \cap dch^{\mP_Y(\sE)}_{\mP_Y(\sE)}(\sO_Y(n)) \cap dch^{\mP_S(\sE)}_{\mP_Y(\sE)}(j_\ast \pi^{\prime\ast}(\sG))  \cap [\eta \circ \pi^*]) \\
& = \pi^\prime_\ast(\Td(j^\ast \mT_{\mP_S(\sE)/Z}) \cap dch^{\mP_Y(\sE)}_{\mP_Y(\sE)}(\sO_Y(n)) \cap dch^{\mP_S(\sE)}_{\mP_Y(\sE)}(\pi^\ast k_\ast (\sG))  \cap [\eta \circ \pi^*]) \\
& = \pi^\prime_\ast(\Td(j^\ast \mT_{\mP_S(\sE)/Z}) \cap dch^{\mP_Y(\sE)}_{\mP_Y(\sE)}(\sO_Y(n)) \cap dch^{\mP_{S}(\sE)}_{\mP_{Y}(\sE)}(\pi^\ast k_\ast (\sG))  \cap [\pi^\ast] \cdot [\eta^\ast] ) 
\end{align*}
The second to third line is a well known property of the local Chern character, the second to last is from base change. However, by multiplicativity of the Todd class,
\begin{align*}
\Td(j^\ast \mT_{\mP_S(\sE)/Z}) \cap -  & = \Td(j^\ast \pi^\ast \mT_{S/Z}) \cap \Td(j^\ast \mT_{\mP_{S}(\sE)/S}) \cap -\\
& = \Td(\pi^{\prime \ast} k^\ast  \mT_{S/Z}) \cap \Td(j^\ast \mT_{\mP_{S}(\sE)/S}) \cap -  \\
 & = \Td(j^\ast \mT_{\mP_{S}(\sE)/S}) \cap \pi^{\prime \ast} \Td( k^\ast  \mT_{S/Z}) \cap -
\end{align*}
This equality combined with commutativity of $dch$ with flat pullback implies 
\begin{align*}
 \pi^\prime_\ast \tau_{g \circ \pi^\prime}([\sF]) & = \pi^\prime_\ast( dch^{\mP_Y(\sE)}_{\mP_Y(\sE)}(\sO_Y(n)) \cap \Td(j^\ast \mT_{\mP_S(\sE)/Z}) \cap [\pi^{\prime \ast}]  \cdot dch^{S}_{Y}(k_\ast (\sG))  \cap [\eta^\ast]) \\
& = \pi^\prime_\ast( dch^{\mP_Y(\sE)}_{\mP_Y(\sE)}(\sO_Y(n)) \cap \Td(j^\ast \mT_{\mP_{S}(\sE)/S}) \cap \pi^{\prime \ast} \Td( k^\ast  \mT_{S/Z}) \cap [\pi^{\prime \ast}] \cdot dch^{S}_{Y}(k_\ast (\sG)))  \cap [\eta^\ast]) \\
& = \pi^\prime_\ast( dch^{\mP_Y(\sE)}_{\mP_Y(\sE)}(\sO_Y(n)) \cap \Td(j^\ast \mT_{\mP_{S}(\sE)/S}) \cap [\pi^{\prime \ast}] \cdot \tau_g([\sG])
\end{align*}
However, we have
\begin{align*}
\pi^\prime_\ast (dch^{\mP_Y(\sE)}_{\mP_Y(\sE)}(\sO_Y(n)) \cap \Td(\mL_{\mP_Y(\sE)/Y}) \cap [\pi^{\prime *}]) = dch^Y_Y(Sym^n q^\ast \sE) 
\end{align*}
This equality is shown in \cite{fulton_intersection} in the discrete case.  It follows here for general $Y$  by the definition of $\mP_Y(\sE)$ and the definition of pullback of operational classes.

Combining this with the fact $Sym^n q^\ast \sE = k^\ast (Sym^n p^\ast \sE)$ (with $p: S \to B$ the natural structure morphism) and the projection formula, 
\begin{align*}
  \pi^\prime_\ast \tau_{g \circ \pi^\prime}([\sF]) & = \pi^\prime_\ast(dch^{\mP_Y(\sE)}_{\mP_Y(\sE)}(\sO_Y(n)) \cap \Td(\mL_{\mP_Y(\sE)/Y}) \cap [\pi^{\prime *}]  \cdot \tau_g([\sG]))\\
& = dch^Y_Y(Sym^n q^\ast \sE) \cap \tau_g([\sG]) \\
&= \tau_g([\sG \otimes Sym^n q^\ast\sE]) \\
& = \tau_g([\pi_\ast^\prime \sO_Y(n) \otimes_{\sO_Y} \sG])\\
&= \tau_{g}([\pi^\prime_\ast (\sO_y(n)\otimes_{\sO_{\mP_Y(\sE)}} \pi^{\prime \ast}\sG)])\\
& =  \tau_g([\pi^\prime_\ast][\sF]) 
\end{align*}
The case of $f$ proper is now shown.


We now show that $\tau$ commutes with products.  Given $X \xrightarrow{f} Y \xrightarrow{g} Z$, let $i^\prime: X \hookrightarrow U$ and $k^\prime: Y \hookrightarrow  V$ be closed embeddings into Zariski open subsets of some $\mP_B(\sE)$. From the diagram  
\begin{displaymath}
  \xymatrix{ X \ar[r]^(.3){i} \ar[rd]^f & Y \times_B U \times_B V \ar[d]^{\pi^\prime} \ar[r]^j  & Z \times_B V \times_B U \ar[d]^\pi \\ & Y \ar[r]^{k} \ar[rd]^g & Z \times_B V \ar[d]^\eta \\ & & Z}
\end{displaymath}
The morphism $i := \Gamma_k \circ f \times i^\prime$ and $j := k \times id$.  It is clear that this is a commutative diagram and the square is Cartesian. For notational convenience, let $W = Z \times_B V$,  $W^\prime = Z \times_B V \times_B U$,  and $W^{\prime \prime} = Y \times_B V \times_B U$.
\begin{align*}
   \tau_f([\sF]) \cdot \tau_g([\sG]) & = \Td(i^\ast \mT_{W^\second/Y}) \cap dch^{W^{\second}}_X(i_\ast \sF)\cap   [\pi^{\prime *}] \cdot \Td(k^\ast \mT_{W/Z}) \cap  dch^W_Y(k_\ast \sG) \cap  [\eta^*]  \\
& =  \Td(i^\ast \mT_{W^\second/Y}) \cap dch^{W^{\prime \prime}}_X(i_\ast \sF)  \cap \Td(\pi^{\prime \ast} k^\ast \mT_{W/Z}) \cap  dch^{W^\prime}_{W^{\prime \prime}}(\pi^\ast k_\ast \sG) \cap [\eta \circ \pi^*] \\
& = \Td(i^\ast \mT_{W^\second/Y}) \cap dch^{W^{\prime \prime}}_X(i_\ast \sF) \cap  \Td(j^\ast \pi^{\ast} \mT_{W/Z}) \cap dch^{W^\prime}_{W^{\prime \prime}}( j_\ast\pi^{\prime \ast} \sG) \cap [\eta \circ \pi^*] \\
& =  \Td(i^\ast j^\ast \mT_{W^\prime/W}) \cap \Td(i^\ast j^\ast \pi^\ast \mT_{W/Z}) \cap dch^{W^{\prime \prime}}_X(i_\ast \sF) \cap  dch^{W^\prime}_{W^{\prime \prime}}( j_\ast\pi^{\prime \ast} \sG) \cap [\eta \circ \pi^*]\\
& = \Td(i^\ast j^\ast \mT_{W^\prime/Z}) \cap dch^{W^{\prime \prime}}_X(i_\ast \sF) \cap  dch^{W^\prime}_{W^{\prime \prime}}( j_\ast\pi^{\prime \ast} \sG) \cap [\eta \circ \pi^*]
\end{align*}
And we are left showing that for $dch^{W^{\prime \prime}}_X(i_\ast \sF) \cap dch^{W^\prime}_{W^{\prime \prime}}( j_\ast\pi^{\prime \ast} \sG) = dch^{W}_{X}((j \circ i)_\ast \sF \otimes i^\ast \pi^{\prime \ast} \sG)$.  
Thus we are reduced to the following: let $f: X \to Y$, $g: Y \to Z$ be composable closed embeddings and $\sG \in \perf(g)$, $\sF \in \perf(f)$ and $k = g \circ f$ then 
\begin{displaymath}
  dch^{Y}_X(\sF) \cap dch^{Z}_{Y}(\sG) = dch^{Z}_{X}( \sF \otimes f^\ast \sG)
\end{displaymath}

By commutativity with proper morphisms, it suffices to show equality as bivariant classes evaluated on discrete schemes.  Given $h: V \to Z$ with $V$ discrete, and the accompanying fiber diagram
\begin{displaymath}
  \xymatrix{V^\second \ar@{^{(}->}[r]^{f^\prime} \ar[d]^{h^\second} &V^\prime \ar[d]^{h^\prime}  \ar@{^{(}->}[r]^{g^\prime} & V \ar[d]^h \\ X  \ar@{^{(}->}[r]^f & Y \ar@{^{(}->}[r]^g & Z}
\end{displaymath}
it is clear $h^\ast g_\ast \sG \cong g^\prime_\ast h^{\prime \ast} \sG$ is a perfect complex on $\sV$ supported on $\widebar V^\prime$.  Since $V$ is assumed to be quasi-projective and $B$ is smooth over our base field $K$, $\bderive(V)$ is the usual bounded derived category.  The above isomorphism, using the $\infty$-categorical Barr-Beck \cite{higher_algebra}, shows $h^\ast g_\ast \sG \otimes_{\sO_V} - $  is naturally an object of $\bderive(V^\prime)$ (i.e., $\sG$ has a natural $\sO_{V^\prime}$-module action and thus imparts this onto any tensor product).  This gives a natural morphism
\begin{align*}
  K_0(V) \xrightarrow{\otimes h^\ast j_\ast \sG}  K_0( V^\prime) 
\end{align*}
denoting this morphism as $\sigma^V_{V^\prime}(\sG)$ and using similar arguments for $\sF$, we claim $\sigma^{V^\prime}_{V^\second}(\sF) \circ \sigma^V_{V^\prime}(\sG) = \sigma^V_{V^\second}(\sF \otimes f^\ast \sG)$.  In other words, we have an equality as elements in operational bivariant K-theory.   Using standard base change formulas, it suffices to show for $V = Z$.  With this setup, the statement  is obvious since the former is induced by $\sK \to (\sK \otimes_{\sO_Z} \sG)\otimes_{\sO_{Y}} \sF)$ while the latter is $\sK \otimes_{\sO_Z}(\sG \otimes_{\sO_{Y}} \sF)$ for any $\sK \in \bderive(Y)$.  Clearly as $\sO_{V}$-modules these are equivalent; since the $\sO_{V^\second}$-module structure is induced by the $\sO_{V^\second}$-module structure on $\sF$ in both cases, they are equivalent as objects in $\bderive(V^\second)$ as well.

To show they are equivalent as elements of operational derived Chow, we use the following lemma
\begin{lemma}
  Let $f: X \to B$ be the structure morphism (B smooth).  Then $K(f) \cong K_0(X) \cong K_0(\widebar X)$ and $\tau_f\otimes \mQ:K(f)\otimes \mQ \xrightarrow{\cong} \dChQ(f) \cong A(\widebar X)\otimes \mQ$. 
\end{lemma}
\begin{proof}
 Let $g: B \to \spec k$ be the structure morphism, then it is clear $\perf(f) \cong \perf(g \circ f)$.  By Lemma \ref{sameoverafield}, $\perf(f) \cong \perf(\widebar f)$.  For $X = \widebar X$, $\perf(g \circ f)  \cong \mr{D}^b(X)$ and the first statement follows.

The second statement is clear for $X = \widebar X$ since in this case $\tau_f$ is the same as defined in \cite{fulton_intersection}.  The result for non-discrete $X$ follows by the commutativity of $\tau$ with closed embeddings.   
\end{proof}
This lemma shows that as operational theories, operational derived K-theory and operational derived Chow are equivalent over the rationals.  As such, an operational identity holding on operational derived K-theory must hold on operational derived Chow as well.
\end{proof}

\section{Applications}
\label{applications}
\subsection{Grothendieck-Riemann-Roch}

For a Grothendieck-Riemann-Roch formula, we must compare the image of the natural orientation of $K$ (applying $\tau$) with that of $\dCh$.   Following Proposition \ref{dch_explicit}, this is all but a formality.

\begin{theorem}
\label{grr}
  Let $f: X \to Y$ be quasi-smooth.  Then $\tau([\sO_X]) = \Td(\mL_{X/Y}) \cap [f^!]_{vir}$.
\end{theorem}
\begin{proof}
Using a choice of factorization in Lemma \ref{factorizeit} and the multiplicativity of both orientations, it suffices to prove the theorem for quasi-smooth closed embeddings and smooth morphisms.  For smooth morphisms, the theorem follows directly from the definitions (using that $dch^X_X(\sO_X) = ch^{\widebar X}_{\widebar X} (\sO_{\widebar X}) = id \in \dCh(X \to X)$).  The case of quasi-smooth embeddings is Proposition \ref{dch_explicit}.  
\end{proof}

We now will derive the more classical explicit Grothendieck-Riemann-Roch formulas.  We will assume $B = \spec k$, with $k$ a field.  The following diagram will be used in the next two statements:
\begin{displaymath}
  \xymatrix{ X \ar[rr]^f \ar[rd]_p & & Y \ar[ld]^q \\ & \spec k &}
\end{displaymath}
  
First, we derive the homological ``virtual'' Grothendieck-Riemann-Roch formula.
\begin{corollary}
Let $f: X \to Y$ be a quasi-smooth morphism  then one has the commutative diagram
\begin{displaymath}
  \xymatrix{K_0(X) \ar[r]^{\tau_X} & \dChQ(X) \\ K_0(Y) \ar[r]^{\tau_Y } \ar[u]^{\otimes_{\sO_Y} \sO_X} & \dChQ(Y) \ar[u]_{\Td(\mT_{X/Y}) \cap [f^!]_{vir}}}
\end{displaymath}
\end{corollary}
\begin{proof}
The left vertical morphism is induced by the natural product $K([f]) \otimes K([q]) \to K([p])$.  Since $\tau$ commutes with products, $\tau_{[p]}([\sO_X] \cdot [\sF]) = \tau_{[f]}([\sO_X]) \cdot \tau_{[q]}([\sF])$.  However, by Theorem \ref{grr}, this is just $\Td(\mT_{X/Y}) \cap [f^!]_{vir} \cdot \tau_{[q]}([\sF])$
\end{proof}

Now, the cohomological statement.
\begin{proposition}
Let $f: X \to Y$ be a proper morphism with $X$, $Y$ quasi-smooth.  Then
 \begin{displaymath}
ch(f_\ast(\sP)) \cap \Td(\mT_{Y/k}) \cap [Y]_{vir} = f_\ast(ch(\sP)\cap \Td(\mT_{X/k}) \cap [X]_{vir})
\end{displaymath}
\end{proposition}
\begin{proof}
By construction, $\tau_{id_X}([\sP]) =[ch(\sP) \cap -]$ and $\tau_{id_Y}([\sP^\prime]) =[ch(\sP^\prime) \cap -]$ for $\sP \in \perf(X)$ and $\sP^\prime \in \perf(Y)$.  
From the commutativity of $\tau$ with proper pushforward we get the equation 
\begin{align*}
  \tau_{id_Y}([f_\ast][\sP]\cdot[\sO_X])\cap [Y]_{vir} & = [f_\ast] \tau_f([\sP] \cdot [\sO_X]) \cap [Y]_{vir}
\end{align*}
The above comments and the definition of the various operators imply that
\begin{displaymath}
  \tau_{id_Y}([f_\ast][\sP]\cdot[\sO_X])\cap [Y]_{vir} = ch(f_\ast(\sP \otimes \sO_X) \cap [Y]_{vir} = ch(f_\ast(\sP)) \cap [Y]_{vir} 
\end{displaymath}
On the other hand, for the right side, commutativity with products implies
\begin{align*}
[f_\ast] \tau_f([\sP] \cdot [\sO_X]) \cap [Y]_{vir} &=   [f_\ast] \tau_{id_X}([\sP]) \cdot \tau_{f}([\sO_X]) \cap [Y]_{vir} \\
& = [f_\ast]( ch(\sP) \cap \tau_f([\sO_X])) \cap [Y]_{vir} \\
& = [f_\ast](ch(\sP) \cap \Td(\mT_{X/Y}) \cap [f^!]_{vir}) \cap [Y]_{vir}
\end{align*}
and we have the equality 
\begin{displaymath}
  ch(f_\ast(\sP)) \cap [Y]_{vir}  = [f_\ast](ch(\sP) \cap \Td(\mT_{X/Y}) \cap [f^!]_{vir}) \cap [Y]_{vir}
\end{displaymath}
Applying $\Td(\mT_{Y/k}) \cap -$ to both sides yields 
\begin{align*}
  ch^Y_Y(f_\ast([\sP])) \cap \Td(\mT_{Y/k}) \cap [Y]_{vir} & =  [f_\ast](ch([\sP]) \cap \Td(\mT_{X/Y}) \cap [f^!]_{vir}) \cap \Td(\mT_{Y/k}) \cap [Y]_{vir}
\end{align*}
  The projection formula then shows  
\begin{displaymath}
ch(f_\ast(\sP) \cap \Td(\mT_{Y/k}) \cap [Y]_{vir} = f_\ast(ch(\sP) \cap \Td(\mT_{X/Y}) \cap \Td(f^\ast \mT_{Y/k}) \cap [f^!]_{vir}) \cap [Y]_{vir} 
\end{displaymath}
Multiplicativity of the Todd class shows $\Td(\mT_{X/Y})\cap  \Td(f^\ast \mT_{Y/k})  \cap - = \Td(\mT_{X/k}) \cap - $ and therefore 
\begin{align*}
ch^Y_Y(f_\ast(\sP \otimes \sO_X)) \cap \Td(\mL_{Y/k}) \cap [Y]_{vir} =   [f_\ast](ch(\sP) \cap \Td(\mL_{X/k}) \cap [f^!]_{vir}) \cap [Y]_{vir}  
\end{align*}

Lemma \ref{chowsameasfield} ensures that we can think of $[Y]_{vir} \in \dCh(Y)$ or as $[q^!]_{vir} \in \dCh([q])$.   Functorality of virtual Gysin homomorphisms ensure that for $[f^!]_{vir} \in \dCh([f])$ and $[q^!]_{vir} \in \dCh([q])$, then $[p^!]_{vir} = [f^!]_{vir} \cdot [q^!]_{vir} \in \dCh([p]) = \dCh(X)$.  Compatibility of pushforward with products then ensures 
\begin{align*}
f_\ast(ch(\sP) \cap \Td(\mL_{X/k}) ) \cap [Y]_{vir} & = [f_\ast](ch(\sP) \cap \Td(\mL_{X/k}) \cap [f^!]_{vir} ) \cap [q^!]_{vir} \\
& = [f_\ast](ch(\sP) \cap \Td(\mL_{X/k}) \cap [p^!]_{vir}) \\
& = f_\ast(ch(\sP) \cap \Td(\mL_{X/k}) \cap [X]_{vir})
\end{align*}
\end{proof}

\subsection{Comparison with pre-existing results}
\subsubsection{Comparison with Ciocan-Fontanine-Kapranov}
Let $char\;  k = 0$, then the natural equivalence of the (model) category of commutative simplicial $k$-algebras and commutative dg-$k$-algebras allows for overlap with the results in \cite{kapranov_fontanine}.  In particular, we have the following comparison with \cite[Theorem 4.2.3]{kapranov_fontanine}.

Let $X$ be quasi-smooth (over $\spec k$) with structure morphism $p$.  Theorem \ref{grr} shows
\begin{displaymath}
  \tau_{[p]}([\sO_X]) = \Td(\mT_{X/k}) \cap [p^!]_{vir}
\end{displaymath}
as elements of $\dChQ([p])$.  After converting with Lemma \ref{chowsameasfield}, this becomes
\begin{displaymath}
  \tau_{[p]}([\sO_X]) = \Td(\mT_{X/k}) \cap [X]_{vir}
\end{displaymath}
$\tau$ commutes with proper pushforward 
\begin{displaymath}
\iota_{X \ast} \tau_{\widebar X}(\alpha) =\tau_X(\iota_{X \ast} \alpha) 
\end{displaymath}
If $\alpha = \sum (-1)^i[\sH^i(\sO_X)]$, then $\iota_{X\ast} \alpha = [\sO_X]$ so 
\begin{displaymath}
 \iota_{X\ast} \tau_{\widebar X}(\sum (-1)^i[\sH^i(\sO_X)]) = \Td(\mT_{X/k}) \cap [X]_{vir} \in \dChQ(X)
\end{displaymath}
The operator $\iota_{X\ast}$ is an isomorphism (on both the K-theory and Chow theory).  If we consider $[X]_{vir}$ as an element of $\dChQ(\widebar X)$ and apply $\iota^{-1}_{X\ast}$ to this equation, since $\Td(\mT_{X/k}) \cap -$ commutes with proper pushforward we have
\begin{align} \label{kf}
 \tau_{\widebar X}(\sum [\sH^i(\sO_X)]) = \Td(\iota_X^\ast \mT_{X/k}) \cap [X]_{vir} \in A(\widebar X)\otimes \mQ
\end{align}
Moving the invertible operator $\Td(\iota_X^\ast \mT_{X/k})$ to the left side yields the statement of \cite[Theorem 4.2.3]{kapranov_fontanine}.  We remark that the embedding of $X$ into a smooth scheme is inherent in their language.  Since it is not built in our theory, we have the additional verification that this class is, in fact, independent of the embedding used and only depends on the homotopy class of $X$.

Although we will not prove it here, it is our belief that using methods similar to Proposition \ref{dch_explicit}, one can find an explicit comparison with \cite[Theorem 4.3.2]{kapranov_fontanine} without relying on Theorem 4.4.2 loc. cit.

\subsubsection{Comparison with Fantechi-G\"ottsche}

It suffices to show how our theory relates the equality in \cite[Lemma 3.5]{fantechi_gottsche}.  This in turn follows from the perfect complex $\sO_{\widehat X}$ and equation (\ref{kf}).  Let $X \to Y$ be quasi-smooth with $Y$ smooth over $\spec k$.  In \S~\ref{deformation_to_normal_cone} we construct a perfect $\sO_{\dcone}$-module $\sO_{\widehat X}$ supported on $\widebar X \times \mP^1 \subset \dcone$.  Since it is perfect $\sum (-1)^i[\sH^i(\sO_{\widetilde X})] \in K_0(\widebar X \times \mP^1)$.  Let $i_0$ be the natural inclusion of $\{0\} \in \mP^1$ and $i_p$ the natural inclusion for any closed $p \notin \{0, \infty\}$.  $\mA^1$-invariance of $K_0(\widebar X)$ shows $i_p^\ast(\sum (-1)^i[\sH^i(\sO_{\widetilde X})]) = i_0^\ast(\sum (-1)^i[\sH^i(\sO_{\widetilde X})])$ in $K_0(\widebar{X})$. By construction $i_p^\ast(\sum (-1)^i[\sH^i(\sO_{\widetilde X})]) = \sum (-1)^i[\sH^i(\sO_X)]$ while 
\begin{displaymath}
i_0^\ast(\sum (-1)^i[\sH^i(\sO_{\widehat X})]) = \sum (-1)^i[\sH^i(\sO_{\widetilde X})] = \sum (-1)^i[\sT or_{\mV(\iota_X^\ast \mT_{X/Y})}(\sO_{\widebar X}, \sO_{C_{\widebar X} Y})]
\end{displaymath}
The result then follows by applying equation (\ref{kf}).

\appendix
\section{Support material for relative pseudo-coherence}

\begin{lemma}
\label{nicepush}
  If $f: X \to Y$ is an embedding such that $f_\ast \sO_X$ is perfect, then $\sF \in \pscoh(X)$ if and only if $f_\ast \sF \in \pscoh(Y)$.
\end{lemma}
\begin{proof}
For one direction, a simple adaptation of \cite[Lemma 3.20]{Lowrey_moduli} shows that $f_\ast: \perf(X) \to \perf(Y)$.  Since $f$ is a perfect morphism (it is affine), this implies it commutes with colimits.  These two facts combine to show that $f_\ast: \pscoh(X) \to \pscoh(Y)$.  The other implication follows directly from the method of proof in \cite[III Lemma 1.1.1]{SGA6} since if $\sF \in \QC^{(-\infty, 0]}(X)$ is $1$-pseudo-coherent then $f_\ast \sF$ $0$-pseudo-coherent if and only if $H^0(f_\ast \sF)$ is a finite type $\sO_{\widebar{Y}}$-module if and only if $H^0(\sF)$ is a finite type $\sO_{\widebar X}$-module if and only if $\sF$ is $0$-pseudo-coherent.
\end{proof}

\begin{lemma}
\label{closedequiv}
Given the commutative diagram
  \begin{displaymath}
    \xymatrix{ Z_1 \ar[rd]_{f} & X \ar[l]_i \ar[r]^{j} & Z_2 \ar[ld]^{f^\prime} \\ & Y &}
  \end{displaymath}
with horizontal arrows closed and vertical arrows smooth, then $i_\ast \sF$ is pseudo-coherent if and only if $j_\ast \sF$ is pseudo-coherent. 
\end{lemma}
\begin{proof}
This is a direct port of \cite[1.1.4]{SGA6}.  We have a factorization of this diagram 
\begin{displaymath}
    \xymatrix{ & X \ar[ld]_i \ar[rd]^{j} \ar[d]^{i \times j} & \\ Z_1 \ar[rd]_{f} & Z_1 \times_Y Z_2  \ar[r] \ar[l] & Z_2 \ar[ld]^{f^\prime} \\ & Y &}
\end{displaymath}
where $i \times j$ is a closed embedding and the horizontal morphisms are smooth.  Symmetry dictates that we are reduced to showing $i_\ast \sF$ is pseudo-coherent if and only if $(i \times j)_\ast \sF$ is pseudo-coherent.  However, this is just the original case with $f^\prime = id$. 
The question is local, so we have reduced to the case

\begin{displaymath}
  \xymatrix{ & Z \ar[d]^f \\ X \ar[r]^i \ar[ur]^j & Y}
\end{displaymath}
where all schemes are affine, $f$ is smooth, and $i, j$ are closed embeddings.  The morphism, being smooth means that $Z$ is finite type over $Y$, and there exists a quasi-smooth embedding $k: Z \to \mA^n \times Y$ (a priori, one just uses the finite type to get an embedding;  analyzing the distinguished triangle arising in the cotangent complexes shows it is quasi-smooth).  By Proposition \ref{localnature}, $k_\ast \sO_Z$ is a perfect complex on $\mA^n \times Y$ and by utilizing Lemma \ref{nicepush}, we can assume that $Z = \mA^n \times Y$ with $j$ a section of $f$ over $X$.  Using that everything is affine, we can extend this to a section $s$ over $Y$:  let $X = \spec A$ and $Y = \spec B$,  then
\begin{align*}
  \Map(X, Z) \cong & \Map(B[x_1, \ldots, x_n], A)\\ & \cong \Map(Sym \oplus^n B, A) \\ & \cong \Map_{B\textrm{-mod}}(\oplus^n B, A)
\end{align*}
Likewise for $\Map(Z, Y)$.  Thus, an extension corresponds to a lift in $B$-mod:
\begin{displaymath}
  \xymatrix{ & \oplus^n B \ar[d] \ar@{-->}[dl] \\ B \ar[r] & A}
\end{displaymath}
This lift is possible since $\pi_0(\cofib (B \to A)) \cong 0$ and $\oplus^n B$ is projective.

The result then follows from Lemma \ref{nicepush} and the fact that the section is quasi-smooth.
\end{proof}

\begin{lemma}
\label{ifandonlyif}
Let the following diagram be Cartesian, 
 \begin{displaymath}
    \xymatrix{X^\prime\ar[d]^{f} \ar[r]^{i^\prime} & Y^\prime \ar[d]^{g} \\ X \ar[r]^{i} &Y}
  \end{displaymath}
with $i$ a closed immersion and $g$  $fppf$.  Then $i_{\ast}\sF \in \pscoh(Y)$ if and only if $i^{\prime}_{\ast} f^\ast \in \pscoh(Y^\prime)$.
\end{lemma}
\begin{proof}
Since closed immersions are proper, we can apply base change and have $g^\ast i_\ast \sF \cong i^\prime_\ast f^\ast \sF$. This shows the if implication since $f^\ast$ preserves pseudo-coherence.  The only if follows by proper base change and \cite[Lemma 4.2]{Lowrey_moduli}.    
\end{proof}

\bibliography{./biblio}
\bibliographystyle{halpha}

\end{document}